\documentclass[12pt,twoside]{amsart}
\usepackage{mabliautoref}
\usepackage{amssymb,amsthm,amsmath}
\usepackage{stmaryrd}
\RequirePackage[dvipsnames,usenames]{xcolor}
\usepackage{hyperref}
\usepackage{mathtools}
\usepackage{comment}
\usepackage[abbrev,alphabetic]{amsrefs}
\usepackage[all]{xy}
\usepackage{tikz}
\usepackage{tikz-cd}

\usepackage{kantlipsum} 

\calclayout

\hypersetup{
	bookmarksdepth=3,
	bookmarksopen,
	bookmarksnumbered,
	pdfstartview=FitH,
	colorlinks,
	linkcolor=Sepia,
	anchorcolor=BurntOrange,
	citecolor=MidnightBlue,
	citecolor=OliveGreen,
	filecolor=BlueViolet,
	menucolor=Yellow,
	urlcolor=OliveGreen
}

\newcommand{\bR}{\mathbf{R}}
\newcommand{\bQ}{\mathbf{Q}}
\newcommand{\bZ}{\mathbf{Z}}
\newcommand{\bN}{\mathbf{N}}
\newcommand{\bP}{\mathbf{P}}

\newcommand{\bB}{\mathbf{B}}

\newcommand{\cB}{\mathcal{B}}

\newcommand{\cJ}{\mathcal{J}}
\newcommand{\cO}{\mathcal{O}}
\newcommand{\cY}{\mathcal{Y}}
\newcommand{\cX}{\mathcal{X}}
\newcommand{\cC}{\mathcal{C}}
\newcommand{\cZ}{\mathcal{Z}}
\newcommand{\cS}{\mathcal{S}}

\newcommand{\frX}{\mathfrak{X}}
\newcommand{\frY}{\mathfrak{Y}}

\newcommand{\frL}{\mathfrak{L}}

\newcommand{\supp}{\mathrm{Supp}}

\newcommand{\lct}{\textup{lct}}
\newcommand{\im}{\textup{Im}}
\newcommand{\codim}{\textup{codim}}

\newcommand{\Exc}{\textup{Exc}}

\newcommand{\coeff}{\textup{coeff}}

\newcommand{\charac}{\textup{char}}

\newcommand{\Can}{\textup{c}}
\newcommand{\Min}{\textup{m}}
\newcommand{\Amp}{\textup{a}}
\newcommand{\reg}{\textup{reg}}
\newcommand{\sing}{\textup{sing}}
\newcommand{\res}{\textup{res}}
\newcommand{\m}{\textup{m}}
\newcommand{\vvert}[1]{\lVert #1 \rVert}

\makeatletter
\newcommand{\xdashrightarrow}[2][]{\ext@arrow 0359\rightarrowfill@@{#1}{#2}}
\def\arrowfill@@#1#2#3#4{%
  $\m@th\thickmuskip0mu\medmuskip\thickmuskip\thinmuskip\thickmuskip
   \relax#4#1
   \xleaders\hbox{$#4#2$}\hfill
   #3$%
}
\makeatother

\DeclareMathOperator{\Chow}{Chow}
\DeclareMathOperator{\Univ}{Univ}
\DeclareMathOperator{\vol}{vol}
\DeclareMathOperator{\Spec}{Spec}
\DeclareMathOperator{\Proj}{Proj}
\DeclareMathOperator{\Spf}{Spf}
\DeclareMathOperator{\relProj}{\textbf{\textup{Proj}}}

\numberwithin{equation}{section}

\begin{document}

\title{Invariance of plurigenera and good minimal models}

\author{Iacopo Brivio}
\address{Center of Mathematical Sciences and Applications, Harvard University, 20 Garden St, Cambridge, MA 02138, USA}
\email{ibrivio@cmsa.fas.harvard.edu}






\begin{abstract}
    Nakayama showed that deformation invariance of plurigenera for smooth complex varieties follows from the MMP and Abundance Conjectures. We generalize his result to families of singular pairs over DVRs of positive or mixed characteristic. As an application we show invariance of plurigenera for deformations of good minimal models of general type varieties whose canonical model has rational singularities as well as boundedness of Gorenstein canonical models of dimension three and fixed volume over an algebraically closed field of characteristic $p>5$.
\end{abstract}


\maketitle

\setcounter{tocdepth}{1}
\tableofcontents


\section{Introduction}

Given a 1-parameter smooth projective family of complex varieties $\pi\colon X\to C$, it was shown by Siu (\cites{Siu_IPGT,Siu_IPNGT}) that the plurigenera of the fibers $h^0(X_c,mK_{X_c})$ are independent of $c\in C$ for all $m\geq 0$. Equivalently, every $m$-pluricanonical form on a fiber can be extended to an $m$-pluricanonical form on $X$. Siu's result, and the techniques introduced in its proof, have had a great impact in the development of higher-dimensional complex algebraic geometry. Notable examples include the proofs of finite generation of the canonical ring (\cites{BCHM,Siu_FGCRAM}) and of boundedness of moduli of general type varieties (\cite{HMX}). It is worth noting that Siu's proof uses analytic techniques which have no algebraic analogue, except for the case when $\pi\colon X\to C$ is a family of complex manifolds of general type (\cite{KawDCS}). On the other hand, Nakayama showed that invariance of plurigenera for smooth complex varieties follows from the MMP and Abundance Conjectures (\cite{Nak}).

In this paper we present a generalization of Nakayama's result to families of singular pairs over bases of any characteristic (zero, positive, or mixed).

\begin{theorem}\label{t-MMP+ABU=>AIP/p (charfree)}
    Assume the Minimal Model Program and the Abundance Conjectures (\autoref{c-MMP} and \autoref{c-ABU}). Let $S$ be either an integral smooth affine curve over a field of characteristic zero or the spectrum of a DVR of positive or mixed characteristic, and let $s,\eta\in S$ be a closed point and the generic point, respectively. Let $\pi\colon (X,\Delta)\to S$ be a projective family of klt pairs such that, if $V\subset \mathbf{B}_-(K_X+\Delta)$ is a non-canonical center of $(X,\Delta+X_s)$, then $V$ is $\pi$-horizontal. Then:
    \begin{itemize}
        \item If $S$ is of equicharacteristic zero, then the restriction map
        $$H^0(X,m(K_X+\Delta))\to H^0(X_s,m(K_{X_s}+\Delta_s))$$
        is surjective for all $m\geq 0$ divisible enough. 
        \item If $S$ is of equicharacteristic $p>0$, or mixed characteristic $(0,p>0)$, then there exists an $e\geq 0$ such that, for all $m\geq 0$ divisible enough, the subspace
        $$H^0(X_s,m(K_{X_s}+\Delta_s))^{p^e}\coloneqq \lbrace s^{p^e} \vert s\in H^0(X_s,m(K_{X_s}+\Delta_s))\rbrace $$
        is contained in the image of the restriction map
        $$H^0(X,mp^e(K_X+\Delta))\to H^0(X_s,mp^e(K_{X_s}+\Delta_s)).$$
    \end{itemize}
    In particular, in both cases, we have $\kappa(X_s,K_{X_s}+\Delta_s)=\kappa(X_\eta,K_{X_\eta}+\Delta_\eta)$.
\end{theorem}

This result is sharp. The condition on the non-canonical centers cannot be dropped, even for families of complex varieties, as shown in \cite{Kaw_OEPPF}*{Example 4.3}. As for positive and mixed characteristic bases, we show in \autoref{s-example} that \autoref{t-MMP+ABU=>AIP/p (charfree)} is the best one can hope for, even for projective families of smooth good minimal models.

In order to obtain positive and mixed characteristic analogues of Siu's theorem, some additional conditions need to be imposed. Previous results (\cites{EH,BBS}) established invariance of plurigenera for families of surface pairs $\pi\colon (X,\Delta)\to S$ which are either of general type or of Kodaira dimension one and $\Delta_s$ is big over $\Proj (R(X_s,K_{X_s}+\Delta_s))$. The following is a higher-dimensional generalization of the aforementioned results.

\begin{theorem}\label{t-extension_R^1_vanishing}
    Let $R$ be either a local Artin ring or a DVR, and suppose the closed point $s\in \Spec(R)$ is perfect of characteristic $p>0$. Let $\pi\colon (X,\Delta)\to\Spec(R)$ be either an infinitesimal deformation\footnote{See \autoref{d-infinitesimal_deformation_pair}} of the pair $(X_s,\Delta_s)$ or a projective family of pairs, respectively. Assume that $K_{X_s}+\Delta_s$ is semiample and that its semiample contraction $f_s\colon X_s\to Y_s$ satisfies $R^1f_{s,*}\cO_{X_s}=0$. Then the restriction map
    \[
    H^0(X,m(K_{X/R}+\Delta))\to H^0(X_s,m(K_{X_s}+\Delta_s))
    \]
    is surjective for all $m\in m_0\bN$, where $m_0$ depends just on $(X_s,\Delta_s)$.
    In particular, this holds when 
    \begin{enumerate}
        \item $K_{X_s}+\Delta_s$ is $\bZ_{(p)}$-Cartier,  $f_s$ is induced by $n(K_{X_s}+\Delta_s)$ for some $n\geq 1$ not divisible by $p$, $\Delta_s$ is $f_s$-ample, $Y_s$ is smooth, and $f_s$ is locally $F$-split\footnote{See \autoref{s-Fsplit}.}; or
        \item $K_{X_s}+\Delta_s$ is big, and $Y_s$ has rational singularities.
    \end{enumerate}
\end{theorem}

\begin{corollary}\label{c-extension_gt_3folds}
    Let $S$ be the spectrum of a DVR with perfect closed point $s$ of characteristic $p>5$. Let $\pi\colon (X,\Delta)\to S$ be a projective family of klt threefold pairs, such that $\Delta$ has standard coefficients, $X$ is $\bQ$-factorial, and $(X_s,\Delta_s)$ is of general type. When $S$ is of equicharacteristic $p$, we also assume $S$ is the spectrum of the complete local ring of a smooth curve $C$ and $(X,\Delta)=(\cX,\Phi)\times_C S$ for some $\bQ$-factorial four-dimensional klt pair $(\cX, \Phi)$, which is projective over $C$. Further, suppose that log resolutions of all log pairs with the underlying variety being birational to $X$ exist (and are constructed by a sequence of blow-ups along the non-snc locus). Lastly, assume that if $V\subset \mathbf{B}_-(K_X+\Delta)$ is a non-canonical center of $(X,\Delta+X_s)$, then $V$ is $\pi$-horizontal. Then the restriction map
    \[
    H^0(X,m(K_X+\Delta))\to H^0(X_s,m(K_{X_s}+\Delta_s))
    \]
    is surjective for all $m\geq 0$ divisible enough.
\end{corollary}

Recall that a scheme $Y$ is said to have \textit{rational singularities} if 
\begin{itemize}
    \item it is excellent, Cohen-Macaulay, normal, with dualizing complex; and
    \item for all locally projective birational morphisms $f\colon X\to Y$ with $X$ an excellent, Cohen-Macaulay, normal scheme the map $\cO_Y\to Rf_*\cO_X$ is a quasi-isomorphism.
\end{itemize}
Strongly $F$-regular (and more generally $F$-rational) singularities are rational by \cite{Smith}. We point out that $n$-dimensional klt singularities in characteristic $p>0$ are expected to be rational when $p$ is large with respect to $n$, while in characteristic zero klt implies rational. For 3-folds over perfect fields, we know this is the case as soon as $p>5$ by \cite{ABL}*{Corollary 1.3}. Combining \autoref{t-extension_R^1_vanishing} with the results in \cites{DW,zhang} we obtain the following boundedness result.

\begin{theorem}\label{t-bounded_3fold_canonical_models'}
Let $k$ be a field of characteristic $p>5$, let $v\in\bQ_{>0}$, and let $\cC_v(k)$ be the set of all three-dimensional varieties $X$ over $k$ such that $K_X$ is ample, $X_{\bar{k}}$ has canonical singularities, $\vol(K_X)=v$, and either $h^1(X,\cO_X)\neq 0$ or $X_{\bar{k}}$ has a Gorenstein terminalization. Then $\cC_v(k)$ is bounded.
\end{theorem}

\autoref{t-bounded_3fold_canonical_models'} has the following arithmetic application.

\begin{corollary}\label{c-rational_points}
    With notation as in \autoref{t-bounded_3fold_canonical_models'}, there exists a positive integer $d$, depending only on $k$ and $v$ such that for all $X\in \cC_v(k)$ there exists an extension $K\supset k$ with $[K:k]\leq d$ and $X(K)\neq \emptyset$. 
\end{corollary}

The paper is organized as follows. \autoref{s-preliminaries} is mostly a miscellaneous recap of fairly well-known notions (Frobenii and universal homeomorphisms, pairs and their singularities, the Minimal Model Program and Abundance Conjectures, and boundedness). The main novelty is the notion of $TNR_1$ pair (\autoref{d-TNR_1pair}), a special class of non-normal pairs which appears naturally in positive and mixed characteristic. The theme of \autoref{s-relative_iitaka_fibration} is understanding to what extent does forming the Iitaka fibration of a semiample divisor on a family of normal varieties commute with restriction to a closed fiber. In general, Iitaka fibration and restriction do not commute with each other, however we can understand this failure in terms of the Stein factorization of the restriction (\autoref{l-finitepartSFconnfibisUH}, \autoref{l-stein_invariance+}). In equicharacteristic zero we answer this question affirmatively in the case of the log canonical divisor of a family of klt pairs (\autoref{t-normalimageK_X+Dnonpositivefibration}). Under a suitable vanishing hypothesis we can show an even stronger characteristic-free result, which allows us to lift semiampleness from the central fiber of an infinitesimal deformation (\autoref{t-formallift_(X,L)_Lsa}). Similarly to \autoref{s-relative_iitaka_fibration}, \autoref{s-MMPfam} studies to what extent does running the MMP in a family of pairs commute with restriction to a closed fiber. The main result, \autoref{t-MMP-topfamilies}, answers this question affirmatively in the case of a family of characteristic zero klt pairs, while in general this holds only ``up to universal homeomorphism''. \autoref{s-proofs} contains the proofs of \autoref{t-MMP+ABU=>AIP/p (charfree)}, \autoref{t-extension_R^1_vanishing}, \autoref{c-extension_gt_3folds}, \autoref{t-bounded_3fold_canonical_models'}, and \autoref{c-rational_points}. \autoref{s-plurigenera_smooth_complex_varieties} is independent from the rest of the paper; there we show that, for projective families of smooth complex varieties, invariance of plurigenera is equivalent to ``asymptotic'' invariance of plurigenera (i.e. $h^0(X_c,mK_{X_c})$ being independent of $c$ for all $m\geq 0$ divisible enough). Lastly, in \autoref{s-example} we present an example of a positive/mixed characteristic projective family of smooth 3-folds violating asymptotic invariance of plurigenera.

\subsection{Proof(s) outline} The proof idea of \autoref{t-MMP+ABU=>AIP/p (charfree)} is fairly simple. Suppose $X_0$ is a smooth projective variety over a field $k$ such that $|K_{X_0}|_{\bQ}\neq \emptyset$. If we assume the MMP and Abundance Conjectures, then we have maps
\begin{center}
    \begin{tikzcd}
        X_0\arrow[r,"\varrho_0",dashed] & X_0^\Min\arrow[r,"g_0"] & X_0^\Amp,
    \end{tikzcd}
\end{center}
where $\varrho_0$ is a (good) minimal model, and $g_0$ is the ample model of $X_0^\Min$. The key feature of this construction is that on $X_0^\Amp$ there is an ample $\bQ$-Cartier $\bQ$-divisor $A_0$ such that $h^0(X_0,mK_{X_0})=h^0(X_0^\Amp,mA_0)$ for all $m\geq 0$ divisible enough. As sections of sufficiently ample line bundles can always be extended from a closed fiber to the total space, simply by Serre vanishing, the upshot is now to show that this construction can be performed in a family. That is, if $\pi\colon X\to C$ is, say, a projective family of smooth varieties over a regular one-dimensional scheme, and 
\begin{center}
    \begin{tikzcd}
        X\arrow[r,"\varrho",dashed] & X^\Min\arrow[r,"g"] & X^\Amp
    \end{tikzcd}
\end{center}
are the minimal, resp., ample model of $X$ \textit{over C}, we want to show that the base change of the above chain of maps to any closed point $c\in C$
\begin{center}
    \begin{tikzcd}
        X_c\arrow[r,"\varrho_c",dashed] & X_c^\Min\arrow[r,"g_c"] & X_c^\Amp
    \end{tikzcd}
\end{center}
yields the minimal, resp., ample model of $X_c$. As mentioned in the introduction, this can be arranged in equicharacteristic zero, but it fails over bases of positive or mixed characteristic. However, this compatibility with base change always holds up to universal homeomorphism and, since universal homeomorphisms in characteristic $p>0$ always factor a (power of) Frobenius, that is why we can always extend sections up to a \textit{fixed} $p^{\textup{th}}$-power.

\autoref{t-extension_R^1_vanishing} is a straightforward consequences of an abstract extension result, \autoref{t-formallift_(X,L)_Lsa}, plus the Kodaira vanishing theorem for $F$-split varieties. \autoref{c-extension_gt_3folds} follows from the same strategy of \autoref{t-MMP+ABU=>AIP/p (charfree)} combined with \autoref{t-extension_R^1_vanishing} and results from \cite{ABL} and \cite{HW}.

The proof of \autoref{t-bounded_3fold_canonical_models'} follows the same ideas as the characteristic zero case (e.g. \cite{HMX_surv}*{1.3}), although some steps become a bit more technical due to the failure of generic smoothness. Lastly, \autoref{c-rational_points} is a straightforward consequence of boundedness.\\

\noindent
\textbf{Acknowledgments}: I would like to thank my PhD advisor, Prof. James M\textsuperscript{c}Kernan, for his support and guidance. I have also greatly benefitted from discussions with Prof. Jungkai A. Chen, Fabio Bernasconi, Justin Lacini, and Liam Stigant to whom I extend my gratitude. The author has been supported by NSF research grants no: 1265263 and no: 1802460, by a grant from the Simons Foundation \# 409187, by the National Center for Theoretical Sciences, a grant from the Ministry of Science and Technology; grant number MOST-110-2123-M-002-005, and by the Center of Mathematical Sciences and Applications.


\section{Preliminaries}\label{s-preliminaries}


\subsection{Notation and conventions}

All of our schemes will be separated and of finite type over a field or a DVR. We will always use $S$ to denote either the spectrum of a DVR of positive or mixed characteristic, or an integral smooth affine curve over a field of characteristic zero. We will denote by $s,\eta\in S$ a closed point and the generic point, respectively. When working in positive characteristic, we will always assume the base field to be $F$-finite.  In practice, almost all of the schemes we will deal with will be integral, $R_1$, and often (albeit not always) $S_2$. For any scheme $X$ we will denote by $X_\reg\subset X$ the \textit{regular locus of $X$} and by $X_{\sing}\coloneqq X\setminus X_{\reg}$ the \textit{singular locus of $X$}. If $X$ is pure-dimensional, an open subset $U\subset X$ is said to be \textit{big} if $\codim_X(X\setminus U)\geq 2$. The normalization morphism will be denoted by $\nu\colon X^\nu\to X$. If $D$ is a $\bQ$-divisor on $X$ such that no codimension one component of $X_{\sing}$ is contained in $\supp(D)$, we denote by $D^\nu$ the divisorial part of $\nu^{-1}D$. Note that if $mD$ is Cartier then $D^\nu=\nu^*(mD)/m$. A morphism of schemes $f\colon X\to Y$ is a \textit{contraction} if it is proper and $f_*\cO_X=\cO_Y$. In particular, contractions have geometrically connected fibers. Most of the times our $X$ will be a normal scheme, in which case $Y$ will be normal as well. If $\pi\colon X\to C$ is a proper morphism of schemes with $\dim(C)=1$, and $Z\subset X$ is a closed subset, we say $Z$ is \textit{$\pi$-horizontal} (resp. \textit{$\pi$-vertical}) if $\pi(Z)=C$ (resp. if $\pi(Z)\subsetneq C$). If $f\colon X\to Y$ and $Z\to Y$ are morphisms of schemes, we will denote by $f_Z\colon X_Z\coloneqq X\times_Y Z\to Z$ the morphism induced by base change.


\subsection{\texorpdfstring{$F$}{F}-split morphisms}\label{s-Fsplit}

Let $k$ be a field of characteristic $p>0$ and let $X$ be a $k$-scheme. We denote by $F_X^e\colon X^{e}\to X$ the ($e^{\textup{th}}$-power of the) \textit{absolute Frobenius of X}. This morphism is the identity on topological spaces and raises functions to the $(p^e)^{\textup{th}}$-power. Note that $X$ and $X^e$ are abstractly isomorphic as schemes, but not as $k$-schemes, in general. Let now $f\colon X\to Z$ be a morphisms of $k$-schemes. Then we have the following commutative diagram for all $e\geq 1$

\begin{equation}\label{e-relfrob}
    \begin{tikzcd}
        X^e \arrow[ddrr, "f^e", bend right=20] \arrow[drr, "F^{e}_{X/Z}"] \arrow[drrrr, bend left=20, "F^{e}_{Z}"] & &  \\
         & &X_{Z^e} \arrow[d, "f_{Z^e}"] \arrow[rr, "(F^{e}_{Z})_{X}"] & & X \arrow[d, "f"]\\
         & & Z^e \arrow[rr, "F^{e}_{Z}"] & & Z.
    \end{tikzcd}
\end{equation}

The morphism $F^e_{X/Z}$ is the \textit{relative Frobenius of $X$ over $Z$} (or \textit{$Z$-linear Frobenius}).

\begin{definition}
    The morphism $f\colon X\to Z$ is \textit{$F$-split} if the natural map $\cO_{X_{Z^e}}\to F^e_{X/Z,*}\cO_{X^e}$ is split in the category of $\cO_{X_{Z^e}}$-modules, for some $e\geq 1$. When $Z=\Spec(k)$ for a perfect field $k$, we say $X$ is $F$-split. Note that this is equivalent to asking the map $\cO_{X}\to F^e_{X,*}\cO_{X^e}$ to be split in the category of $\cO_{X}$-modules.
\end{definition}


\subsection{Universal homeomorphisms}

A morphism of schemes $f\colon X\to Y$ is a \textit{universal homeomorphism} if, for all $Y'\to Y$, the base change $f_{Y'}\colon X_{Y'} \to Y'$ is a homeomorphism. It is known that a finite morphism of $k$-schemes is a universal homeomorphism if and only if it factors a power of the $k$-linear Frobenius morphism (\cite{Kol_QuotSpaces}*{Proposition 6.6}). As a consequence we have the following.

\begin{lemma}[{\cite[Lemma 1.4]{Keel}}]\label{l-keelfrobeniusdescent}
    Let $h\colon Z\to Y$ be a finite universal homeomorphism between schemes over a field of characteristic $p > 0$. Then there exists $q=p^e$ such that, whenever $L$ is a line bundle on $Y$ and $\sigma$ is a global section of $h^*L$ on $Z$, then $\sigma^q$ descends to $Y$; that is, there exists a global section $\tau$ of $L^q$ on $Y$ such that $\sigma^q=h^*\tau$.
\end{lemma}

\begin{corollary}\label{c-QcartsmallUH}
    Let $h\colon Z\to Y$ be a finite universal homeomorphism between integral $R_1$ schemes over a field of characteristic $p > 0$. Let $D$ be a $\bQ$-divisor on $Y$ and let $D_Z$ be the divisorial part of $h^{-1}D$. Then $D$ is $\bQ$-Cartier if and only if $D_Z$ is $\bQ$-Cartier.
\end{corollary}

\begin{proof}
    If $mD$ is Cartier then clearly so is $mD_Z=h^*(mD)$. Suppose now that $mD_Z$ is Cartier and let $f$ be a local equation. By \autoref{l-keelfrobeniusdescent} we know that $f^{p^e}=h^*g$ for some rational function on $X$ and some $e\geq 0$. Again, $h$ is an isomorphism in codimension one, hence the divisor defined by $g$ coincides with $mp^eD$, thus $D$ is $\bQ$-Cartier.
\end{proof}

Universal homeomorphisms appear naturally in the Stein factorization of morphisms with connected fibers.

\begin{lemma}\label{l-finitepartSFconnfibisUH}
    Let $k$ be any field and let $f\colon X\to Y$ be a proper morphism of integral $k$-schemes with connected geometric fibers. Let
    \[
    f\colon X\xrightarrow{g} Z\xrightarrow{h}Y
    \]
    be its Stein facorization. Then $h$ is a finite universal homeomorphism. If $X$ is normal and either $\charac(k)=0$ or $f$ is birational, then $h$ is an isomorphism if and only if $Y$ is normal.
\end{lemma}

\begin{proof}
     By \cite[Corollaire 18.12.11]{EGAIV_4}, in order to show that $h$ is a universal homeomorphism it is enough to check that $h$ is radicial, that is for every algebraically closed field $K\supset k$ the map $Z(K)\to Y(K)$ is injective (see \cite[Definition 3.5.4]{EGA1}). The morphism $\Spec(K)\to \Spec (k)$ is flat, hence 
    \[f_{K}\colon X_{K}\xrightarrow{g_{K}}Z_{K}\xrightarrow{h_{K}}Y_{K}\]
    is the Stein factorization, as its formation commutes with flat base change (\cite[Tag 03GX]{SP}). As $f_{K}$ has connected geometric fibers, $Z(K)\to Y(K)$ is injective. 

    Suppose now that $X$ is normal and that either $\charac(k)=0$ or $f$ is birational. Then $Z$ is normal too and in both cases we have that $h$ is a finite birational universal homeomorphism. We conclude by Zariski's Main Theorem (\cite[Lemme 8.12.10.I]{EGAIV_3}).
\end{proof}

\begin{definition}
    A finite universal homeomorphism of schemes is \textit{small} if it is an isomorphism in codimension one.
\end{definition}

\begin{proposition}\label{p-AmpleUH}
    Let $h\colon X\to Y$ be a finite universal homeomorphism of proper $k$-schemes, and let $A$ be an ample Cartier divisor on $Y$. Then $h$ is an isomorphism if and only if $h^*\colon H^0(Y,mA)\to H^0(X,h^*(mA))$ is an isomorphism for all $m\geq 0$.
\end{proposition}
\begin{proof}
    We have a short exact sequence
    \[
    0\to \cO_Y(mA)\to h_*\cO_X(h^*(mA))\to Q(mA)\to 0,
    \]
    where $Q=0$ if and only if $h$ is an isomorphism. For $m$ large enough, all the above sheaves will be generated by global sections and have vanishing cohomology in positive degree. Hence the claim follows by taking global sections.
\end{proof}


\subsection{Varieties, pairs, and their families}

\begin{definition}
Let $B$ be a regular integral affine scheme, and let $X$ be an $R_1$ scheme, quasi-projective over $B$. Let $j\colon X_{\reg}\to X$ be the natural inclusion, let $i\colon X_\reg\hookrightarrow\bP^N_B$ be a locally closed embedding, and let $I_X$ be the defining ideal of the closure of $i(X_\reg)$. Then we set
\[
\omega_{X_\reg/B}\coloneqq i^*(\omega_{\bP^N_B/B}\otimes (\det(I_{X}/I_{X}^2))^\vee).
\]
The \textit{canonical sheaf of $X$ over $B$} is then 
\[
\omega_{X/B}\coloneqq j_*\omega_{X_\reg/B}.
\]
Note that $\omega_{X/B}$ is a torsion-free $S_2$ sheaf of rank one by construction. The corresponding linear equivalence class of (Weil) divisor is denoted by $K_{X/B}$ and it is called the \textit{canonical divisor of $X$ over $B$}. In particular, when $X$ is normal, $\omega_{X/B}$ is reflexive. In this paper $B$ will typically be the spectrum of a field or a DVR, hence we will usually drop the ``$/B$'' from $\omega_{X/B}$ and $K_{X/B}$.
\end{definition}

\begin{definition}
    Let $k$ be any field. A \textit{k-variety} is a geometrically integral scheme which is quasi-projective over $k$. A \textit{topologically normal $R_1$} ($TNR_1$ from now on) \textit{$k$-variety} is a $k$-variety $X$ which is $R_1$ and whose normalization $\nu\colon X^\nu\to X$ is a small universal homeomorphism. Note that $\nu$ can also be described as the demi-normalization (or $S_2$-ification) of $X$, as in \cite[Definition 5.1]{Kol_SMMP}. We will drop the ``$k$'' when the base field is clear.
\end{definition}

\begin{remark}\label{r-advantage_variety}
    Our definition of variety has the advantage of being stable under field extension. This is particularly useful when the base field is imperfect. In particular, a $k$-variety $X$ will always be generically \textit{smooth} over $k$. Another advantage is the following: suppose that $\pi\colon X\to C$ is a flat contraction from an integral normal scheme to a regular one-dimensional scheme, where the fibers are $TNR_1$ varieties. Then there exists a big regular open subset $U\subset X$ such that $U\to C$ is surjective, $\codim_{X_c}(X_c\setminus U_c)\geq 2$ for all $c\in C$, and the induced morphism $\pi|_U\colon U\to C$ has regular fibers. In particular, $K_{X/C}\vert_{X_c}=K_{X_c}$ for all  $c\in C$. On the other hand, even when $f\colon X\to Y$ is a morphism of \textit{smooth} varities over a perfect field of characteristic $p>0$, it can happen that a general fiber $X_y$ (or the generic fiber $X_\eta$) is not a variety. Compare with \cite[Definition 2.1]{DW}.
\end{remark}

\begin{definition}
    Let $B$ be a regular integral affine scheme. We call $(X,\Delta)$ a \textit{pair over B} if 
    \begin{itemize}
        \item $X$ is an integral, normal, quasi-projective $B$-scheme, and
        \item $\Delta$ is an effective $\bQ$-divisor such that $K_{X/B}+\Delta$ is $\bQ$-Cartier.
    \end{itemize}
    When $B$ is the spectrum of a field, we will additionally require $X$ to be geometrically integral. As usual, we will drop the ``over $B$'' when it is clear from context.
\end{definition}

The following notion arises naturally when working in positive or mixed characteristic.

\begin{definition}\label{d-TNR_1pair}
    Let $B$ be a regular integral affine scheme. We call $(X,\Delta)$ a \textit{$TNR_1$ pair over $B$} if $X$ is an $R_1$ integral scheme whose normalization $\nu\colon X^\nu\to X$ is a small universal homeomorphism, and an effective $\bQ$-divisor $\Delta$ such that $K_{X/B}+\Delta$ is $\bQ$-Cartier. If $B$ is the spectrum of a field we additionally require $X$ to be geometrically integral.
\end{definition}

\begin{remark}\label{r-crepant_pullback_top_R1}
    Let $(X,\Delta)$ be a $TNR_1$ pair over $B$, and let $\nu\colon X^\nu\to X$ be the normalization. It follows from the definition that $(X^\nu,\Delta^\nu)$ is a pair over $B$. Conversely, if $(Y,\Theta)$ is a pair over a positive characteristic field $k$, and $h\colon Y\to \bar{Y}$ is a small birational universal homeomorphism, then $(\bar{Y},\bar{\Theta}\coloneqq h_*\Theta)$ is a $TNR_1$ pair over $k$, by \autoref{c-QcartsmallUH}.
\end{remark}

\begin{definition}
    Let $(X,\Delta)$ be a $TNR_1$ pair over $B$ such that $m(K_X+\Delta)$ is Cartier, let $f\colon Y\to X$ be a proper birational morphism from a normal integral scheme, and let $\lbrace E_i\rbrace_i$ be the set of $f$-exceptional divisors, so that we have an isomorphism of invertible sheaves
    \begin{displaymath}
    \omega_{Y/B}^{[m]}(mf_*^{-1}\Delta)\cong \left(f^*\omega_{X/B}^{[m]}(m\Delta)\right)\left( \sum_ima(E_i,X,\Delta)E_i\right),
    \end{displaymath}
    for some rational numbers $a(E_i,X,\Delta)$. We call $a(E_i,X,\Delta)$ the \textit{discrepancy of $E_i$ with respect to $(X,\Delta)$}. If $D\subset X$ is a prime divisor, we define $a(D,X,\Delta)\coloneqq -\coeff_D(\Delta)$. Discrepancies are a measure of the singularities of $(X,\Delta)$. When $\Delta$ is a boundary we say that $(X,\Delta)$ is 
    \begin{itemize}
        \item terminal if $a(E,X,\Delta)>0$ for every exceptional $E$;
        \item canonical if $a(E,X,\Delta)\geq 0$ for every exceptional $E$;
        \item klt if $a(E,X,\Delta)>-1$ for every $E$;
        \item plt if $a(E,X,\Delta)>-1$ for every exceptional $E$;
        \item dlt if $a(E,X,\Delta)>-1$ for every exceptional $E$ with $\textup{center}_X(E)\subset \textup{non-snc}(X,\Delta)$;
        \item lc if $a(E,X,\Delta)\geq -1$ for every $E$.
    \end{itemize}
    Here $\textup{center}_X(E)$ denotes the image of $E$ in $X$ and $\textup{non-snc}(X,\Delta)$ denotes the non simple normal crossing locus of $(X,\Delta)$. For future reference, we remark that this definition of discrepancy (and the related classes of singularities) is consistent with the discrepancy for non-normal pairs as defined in \cite[Remark 2.6]{FA}. If $V=\textup{center}_X(E)$ for some divisor with discrepancy $<0$, we call $V$ a \textit{non-canonical center of $(X,\Delta)$}.
\end{definition}

\begin{definition}
    Let $C$ be a regular integral affine scheme of dimension one, and let $(X,\Delta)$ be a pair over $C$. We say $\pi\colon (X,\Delta)\to C$ is a \textit{projective family of pairs} if:
    \begin{itemize}
        \item $\pi$ is a projective contraction whose fibers are normal varieties (in particular $\pi$ is flat);
        \item $\pi|_{\Delta_i}\colon\Delta_i\to C$ is flat for every irreducible component $\Delta_i$ of $\supp(\Delta)$.
    \end{itemize}
    In particular, $(X_c,\Delta_c)$ is a pair for all $c\in C$.
    
    We say $\pi\colon (X,\Delta)\to C$ is a \emph{projective family of $TNR_1$ pairs} if:
    \begin{itemize}
        \item $\pi$ is a projective contraction, whose fibers are $TNR_1$ varieties (in particular $\pi$ is flat);
        \item the geometric generic fiber of $\pi$ is normal;
        \item $\pi|_{\Delta_i}\colon\Delta_i\to C$ is flat for every irreducible component $\Delta_i$ of $\supp(\Delta)$.
    \end{itemize}
    In particular, $(X_c,\Delta_c)$ is a $TNR_1$ pair for all $c\in C$ and an actual pair for all $c$ but finitely many.
\end{definition}

$TNR_1$ pairs arise naturally via adjunction on plt centers of characteristic $p>0$.

\begin{lemma}[{\cite[Lemma 2.1]{HW},\cite[Proposition 5.51]{KM}}]\label{l-pltcentersnormalUH}
    Let $(X,\Delta+D)$ be a plt pair, where $D$ is prime and $\bQ$-Cartier. Then the normalization $D^\nu\to D$ is a small universal homeomorphism. If $\bQ\subset \cO_X$, then $\nu$ is an isomorphism.
\end{lemma}

\begin{remark}
Let $C$ be a regular integral affine scheme of dimension one and let $(X,\Delta)$ be a flat pair over $C$, such that $(X,\Delta+X_c)$ is plt for some closed point $c\in C$. In this case, as $X_c$ is Cartier, we have that the different $\textup{Diff}_{X_c}(\Delta)$ equals $\Delta_c$ (see \cite[Proposition 4.5]{Kol_SMMP}). In particular $(X_c,\Delta_c)$ is a $TNR_1$ pair by \autoref{l-pltcentersnormalUH}.
\end{remark}

\begin{remark}
    The definition of family of pairs over a higher-dimensional base is more subtle (\cite[Chapter 4]{Kol_FVGT}). The key issue is how to define the restriction $\Delta|_{X_c}$ as a $\bQ$-divisor. Since we will deal with families over higher-dimensional bases only in the case of empty boundary, this will not trouble us too much.
\end{remark}

\begin{definition}
    Let $B$ be a regular integral affine scheme, and let $X$ be an integral normal scheme. We say that a morphism $\pi\colon X\to B$ is a \textit{projective family of normal $\bQ$-Gorenstein varieties} if $K_X$ is $\bQ$-Cartier and $\pi$ is a flat and projective contraction whose fibers are normal varieties. If $\pi$ is smooth we call it a \textit{projective family of smooth varieties}.
\end{definition}

Next, we briefly recall some notions of lifting and deformations of varieties, line bundles, and pairs over non-reduced bases.

\begin{definition}\label{d-lifting_of_(X,L)}
    Let $X_0$ be an integral normal proper $k$-scheme, and let $L_0$ be a line bundle. A \textit{lifting} of $(X_0,L_0)$ over a scheme $T$ consists of
	\begin{enumerate}
		\item a $T$-scheme $X$ such that $X \to T$ is a flat contraction;
		\item a line bundle $L$ on $X$;
		\item a morphism $\alpha \colon \Spec(k) \to T$ and an isomorphism $\gamma \colon X_0\to X \times_T \Spec(k)$ such that $\gamma^*(L\otimes_T k) = L_0$.
		\end{enumerate}
    Let $(R,\mathfrak{m})$ be a complete local ring, and denote by $\Spf(R)$ the formal completion of $\Spec(R)$ at $\mathfrak{m}$. We say that $(\frX,\frL)$ is a \textit{formal lifting} of $(X,L)$ over $\Spf(R)$ if
    \begin{enumerate}
        \item $\frX \to \Spf(R)$ is formal scheme;
        \item $\frL$ is a line bundle on $\frX$;  
        \item for any $n>0$, the truncation $(X_n, L_n) = (\frX, \frL) \times_{\Spf(R)} \Spec(R/\mathfrak{m}^{n+1})$ is a lifting of $(X_0, L_0)$ over $\Spec(R/\mathfrak{m}^{n+1})$.
    \end{enumerate} 
\end{definition}

\begin{definition}\label{d-infinitesimal_deformation_pair}
    Let $A$ be a local Artinian ring with closed point $s$, let $(X_s,\Delta_s)$ be a projective pair over the residue field $\kappa(s)$, and let $n\geq 1$ be such that $n(K_{X_s}+{\Delta_s})$ is Cartier. We say $\pi\colon (X,\Delta)\to\Spec(A)$ is an \textit{infinitesimal deformation} of $(X_s,\Delta_s)$ if 
    \begin{enumerate}
        \item $X$ is an $A$-scheme such that $X\to\Spec(A)$ is a flat contraction, and there exists an isomorphism $\gamma\colon X\otimes \kappa(s)\cong X_s$;
        \item there exists an $A$-flat closed subscheme $Z\subset X$ such that $X^0\coloneqq X\setminus Z\xhookrightarrow{j} X$ is big and $A$-smooth;
        \item there exists a relative Cartier divisor $D^0$ on $X^0$ such that $\gamma(D^0)=n\Delta$; and
        \item for all $m\in n\bZ$ the sheaf $\omega_{X/A}^{[m]}(m\Delta)\coloneqq j_*\omega_{X^0/A}^m((m/n)D^0)$ is a line bundle on $\widetilde{X}$, and $\omega_{X/A}^{[m]}(m\Delta)\otimes \kappa(s)$ is identified with $\omega_{X_s/s}^{[m]}(m\Delta_s)$ via $\gamma$.
    \end{enumerate}
    Here $\Delta$ is formally identified with $(\textup{the closure of }D^0)/n$. Note that $(X,\omega_{X/A}^{[m]}(m\Delta))$ is a lifting of $(X,\omega^{[m]}_{X_s/s}(m\Delta))$ over $\Spec(A)$ in the sense of \autoref{d-lifting_of_(X,L)}. We will denote by $K_{X/A}+\Delta$ the $\bQ$-Cartier $\bQ$-divisor (class) such that $m(K_{X/A}+\Delta)$ corresponds to $\omega_{X/A}^{[m]}(m\Delta)$.
\end{definition}

\begin{remark}
    Suppose $\pi\colon (X,\Delta)\to S$ is a projective family of pairs over the spectrum of a DVR $(R,\mathfrak{m})$, and let $S_n\coloneqq \Spec(R/\mathfrak{m}^{n+1})$ for all $n\geq 0$. Then the base change $\pi_n\colon (X_n,\Delta_n)\to S_n$ is an infinitesimal deformation of $(X_s,\Delta_s)$ for all $n\geq 0$. 
\end{remark}

\begin{lemma}\label{l-lifting_infinitesimals_implies_lifting_in_family}
    Let $X_0$ be an integral proper $k$-scheme and let $L_0$ be a line bundle on $X_0$. Let $S$ be a DVR with closed point $s$ and let $(X,L)$ be a lifting of $(X_0,L_0)$ over $S$. Suppose that the restriction map
    \[
    H^0(X_n,L_n)\to H^0(X_0,L_0)
    \]
    is surjective for all $n\geq 0$. Then so is the restriction map
    \[
        H^0(X,L)\to H^0(X_0,L_0).
    \]
\end{lemma}
\begin{proof}
    Let $S=\Spec(R)$ and let $\varpi\in R$ be a uniformizer. We want to show surjectivity of $H^0(X_{n+1},L_{n+1})\to H^0(X_n,L_n)$ for all $n\geq 0$ by induction, the case $n=0$ being obvious. Note that we have isomorphisms $L_{i-1}\xrightarrow{\cdot \varpi}\varpi L_{i}$ for all $i\geq 1$ hence, by induction, we have a surjection $H^0(X_{n+1},\varpi L_{n+1})\to H^0(X_n,\varpi L_n)$. Fix now $s_n\in H^0(X_n,L_n)$: we want to find $s_{n+1}\in H^0(X_{n+1},L_{n+1})$ such that $s_{n+1}|_{X_n}=s_n$. By induction we can find $s_{n+1}'$ such that $s_{n+1}'|_{X_0}=s_n|_{X_0}$. In other words $s'_{n+1}|_{X_n}-s_n\in H^0(X_n,\varpi L_n)$. Hence we can find $s_{n+1}''\in H^0(X_{n+1},\varpi L_{n+1})$ such that $s_{n+1}''|_{X_n}=s'_{n+1}|_{X_n}-s_n$. By setting $s_{n+1}\coloneqq s_{n+1}'-s_{n+1}''$ we conclude.
\end{proof}


\subsection{Birational geometry} We refer the reader to \cite{KM} and \cite{Kol_SMMP} for an introduction to the Minimal Model Program. Here we just recall some notions needed in order to make sense of the results proven in this paper.

\begin{definition}
    Let $(X,\Delta)$ be a projective pair over an integral regular affine scheme $B$. We say it is a 
    \begin{itemize}
        \item \textit{minimal model} if it is dlt and $K_X+\Delta$ is nef over $B$;
        \item \textit{good minimal model} if it is a minimal model and $K_X+\Delta$ is semiample\footnote{See \autoref{d-semiample} and \autoref{t-semiamplefibration} for more on semiample divisors.} over $B$, i.e. $m(K_X+B)$ is basepoint free for some $m\geq 1$;
        \item \textit{canonical model} if it is lc and $K_X+\Delta$ is ample over $B$.
    \end{itemize}
    A contraction $f\colon X\to Y$ of $B$-schemes is a $(K_X+\Delta)$\textit{-Mori fiber space} (also called \textit{fiber-type contraction} in \cite{KM}) if $\dim(X)>\dim(Y)$, $K_X+\Delta$ is $f$-antiample, and $X$ has Picard rank one over $Y$.
\end{definition}

\begin{definition}\label{d-D-negative contraction_ample_model}
    A birational map of integral normal quasi-projective $B$-schemes $\phi \colon X \dashrightarrow Y$ is a \emph{birational contraction} if it is proper and $\phi^{-1}$ does not contract any divisor.
    Let $\phi\colon X\dashrightarrow Y$ be a  birational contraction and let $D$ be a $\bQ$-Cartier divisor on $X$ such that $D'\coloneqq \phi_* D$ is also $\bQ$-Cartier. Let $W$ be the normalized closure of the graph of $\phi$ in $X\times_B Y$, and let $p$ and $q$ the induced morphisms to $X$ and $Y$, respectively. We say that $\phi$ is \textit{$D$-non-positive} (respectively \textit{$D$-negative}) if
    \[
    p^*D=q^*D'+E
    \]
    where $E\geq 0$ is $q$-exceptional (respectively $E\geq 0$ is $q$-exceptional and the support of $E$ contains the strict transform of the $\phi$-exceptional divisors).

    Let now $g\colon X\dashrightarrow Z$ be a dominant rational map of integral normal quasi-projective $B$-schemes. Let $V$ be the normalized closure of the graph of $g$ in $X\times_B Z$, and let $p$ and $q$ be the induced morphisms to $X$ and $Z$, respectively. We say $g$ is the \textit{ample model of $D$ over $B$} if $q$ is a contraction and there is a $\bQ$-divisor $H$ on $Z$, ample over $B$, such that we may write $p^*D\sim_{\bQ,B}q^*H+E$, where $E\geq 0$ and $G\geq E$ for every $G\in |p^*D/B|_{\bQ}$.
\end{definition}

\begin{remark}\label{r-non_positive_contractions_preserve_sections}
    Keeping the notation of \autoref{d-D-negative contraction_ample_model}, if $\phi\colon X\dashrightarrow Y$ is a $D$-non-positive birational contraction, the projection formula yields isomorphisms $H^0(X,mD)\cong H^0(Y,mD')$ for all $m\geq 0$ divisible enough. Similarly, if $g\colon X\dashrightarrow Z$ is the ample model of $D$ then the projection formula yields isomorphisms $H^0(X,mD)\cong H^0(Z,mH)$ for all $m\geq 0$ divisible enough.
\end{remark}

An important feature of $(K_X+\Delta)$-non-positive birational contractions is that they preserve singularities.

\begin{lemma}\label{l-negcontimprovesings}
    Let $B$ be a regular integral affine scheme. Consider the commutative diagram
    \begin{equation*}
        \begin{tikzcd}
            X_1\arrow[rr,"\varrho",dashed]\arrow[dr,"f_1",swap] &   & X_2\arrow[dl,"f_2"]\\
              & Z & 
        \end{tikzcd}
    \end{equation*}
    where $f_1$ and $f_2$ are proper birational maps of integral normal $B$-schemes. Let $(X_1,\Delta_1)$ and $(X_2,\Delta_2)$ be pairs over $B$ such that
    \begin{itemize}
        \item $f_{1,*}\Delta_1=f_{2,*}\Delta_2$;
        \item $-(K_{X_1}+\Delta_1)$ is $f_1$-nef;
        \item $K_{X_2}+\Delta_2$ is $f_2$-nef.
    \end{itemize}
    Then for any exceptional divisor $E$ over $Z$ we have
    \[
    a(E, X_1, \Delta_1) \leq a(E, X_2, \Delta_2).
    \]
    In particular, if $(X_1,\Delta_1)$ is klt (resp. lc) then so is $(X_2,\Delta_2)$. The same holds when $(X_1,\Delta_1)$ is plt (resp. dlt), provided $\varrho$ does not contract any component of $\lfloor\Delta_1\rfloor$.
    Furthermore, if $-(K_{X_1} + \Delta_1)$ is $f_1$-ample and $f_1$ is not an isomorphism
    above the generic point of $\textup{center}_{Z}(E)$, or if $K_{X_2} + \Delta_2$ is $f_2$-ample and $f_2$ is not an isomorphism
    above the generic point of $\textup{center}_{Z}(E)$ then $a(E, X_1, \Delta_1) < a(E, X_2, \Delta_2)$. 
    
    The same conclusions hold also when $(X_1,\Delta_1)$ and $(X_2,\Delta_2)$ are $TNR_1$ pairs over a field of characteristic $p>0$.
\end{lemma}

\begin{proof}
    When the $X_i$ are normal, this is \cite[Corollary 1.23]{Kol_SMMP}. 

    When the $(X_i,\Delta_i)$ are $TNR_1$ pairs over a field of characteristic $p>0$ we have a commutative diagram of normalizations
    \begin{equation*}
        \begin{tikzcd}
            (X_1^\nu,\Delta_1^\nu)\arrow[rr,"\varrho^\nu",dashed]\arrow[dr,"f_1^\nu",swap] &   & (X_2^\nu,\Delta_2^\nu)\arrow[dl,"f_2^\nu"]\\
              & Z^\nu, & 
        \end{tikzcd}
    \end{equation*}
    such that $f^\nu_{1,*}\Delta_1^\nu=f^\nu_{2,*}\Delta_2^\nu$, $-(K_{X_1^\nu}+\Delta_1^\nu)$ is $f_1^\nu$-nef, and $K_{X_2^\nu}+\Delta_2^\nu$ is $f_2^\nu$-nef.
    We conclude by applying \cite[Corollary 1.23]{Kol_SMMP} to the above diagram.
\end{proof}

\begin{definition}\label{d-goodminimalcanonical_model}
    Let $(X,\Delta)$ be a projective dlt pair over a regular integral affine scheme $B$ and let 
    \[
    \varrho\colon (X,\Delta)\dashrightarrow (Y,\Gamma\coloneqq \varrho_*\Delta)
    \]
    be a birational contraction over $B$. We say that $\varrho$ (or $(Y,\Gamma)$) is
    \begin{itemize}
        \item a \textit{(good) minimal model of} $(X,\Delta)$ if $K_Y+\Gamma$ is $\bQ$-Cartier, $\varrho$ is $(K_X+\Delta)$-negative, and $(Y,\Gamma)$ is a (good) minimal model;
        \item the\footnote{If the canonical model exists, it is unique (\cite[Theorem 1.26]{Kol_SMMP})} \textit{canonical model} of $(X,\Delta)$ if $K_Y+\Gamma$ is $\bQ$-Cartier, $\varrho$ is $(K_X+\Delta)$-non-positive, and $(Y,\Gamma)$ is a canonical model;
    \end{itemize}        
\end{definition}

The Minimal Model Program predicts the following dichotomy: given a dlt pair $(X,\Delta)$ projective over $B$, there exists a $(K_X+\Delta)$-negative birational contraction 
\[
\varrho\colon (X,\Delta)\dashrightarrow (X',\Delta'\coloneqq\varrho_*\Delta)
\]
such that
\begin{itemize}
    \item if $K_X+\Delta$ is pseudoeffective over $B$, then $(X',\Delta')$ is a minimal model;
    \item if $K_X+\Delta$ is not pseudoeffective over $B$, then there exists a $(K_{X'}+\Delta')$-Mori fiber space, $f\colon X'\to Y$.
\end{itemize}

Furthermore, the map $\varrho$ can be understood in terms of elementary birational transformations, namely \textit{divisorial contractions} and \textit{flips}\footnote{See \cite{KM} for the definition of divisiorial and flip(ping) extremal contraction.}. 

The Abundance Conjecture predicts that every minimal model is actually good.

Over the complex numbers the Minimal Model Program has been established in dimension $\leq 3$, or when $K_X+\Delta$ is big over $B$ (\cites{Mori,BCHM}). In the latter case, Abundance is a consequence of the basepoint-free theorem. In the non-big case, Abundance is known in dimension $\leq 3$ (\cites{Miy_Abu,Miy_min3fold,Kaw_Plurican,KMM}), and when $K_X+\Delta\equiv 0$ (\cites{Kaw_Fiberspaces,Nak_ZDA}). Over a perfect field of characteristic $p>5$, the Minimal Model Program is known in dimension $\leq 3$ (\cites{HXp,Tan_MMPABU}). A good part of it holds even when the field is imperfect (\cites{DW,Waldron}). Abundance is also known to hold in dimension $\leq 3$ (\cites{Tan_MMPABU,DW_abu,Wal_FG,Wal_LMMP,Wit_CBF,XuZheng_Abu}), except the case where $\dim (X)=3$ and $K_X+\Delta$ has numerical Kodaira dimension one. More recently, advances in mixed characteristic commutative algebra have made it possible to establish the Minimal Model Program even for mixed characteristic 3-folds (\cite{BMPSTWW}), and subsequently the Abundance conjecture (\cite{BBS}).

For the purposes of proving \autoref{t-MMP+ABU=>AIP/p (charfree)}, we will assume the following special cases of the Minimal Model Program and Abundance Conjectures for pairs over regular affine schemes of dimension one.

\begin{conjecture}\label{c-MMP}
     Let $S$ be either an integral smooth affine curve over a field of characteristic zero or the spectrum of a DVR of positive or mixed characteristic, and let $s\in S$ be a closed point. Let $\pi\colon (X,\Delta)\to S$ be a projective family of klt pairs such that, if $V\subset \bB_-(K_X+\Delta)$ is a non-canonical center of $(X,\Delta+X_s)$ then $V$ is $\pi$-horizontal. Then there is a sequence of finitely many $(K_X+\Delta)$-negative divisorial contractions and $(K_X+\Delta)$-flips
    \[
    \varrho\colon (X,\Delta)\dashrightarrow (X',\Delta'\coloneqq \varrho_*\Delta)
    \]
    over $S$, such that
    \begin{itemize}
        \item if $K_X+\Delta$ is pseudoeffective, then $(X',\Delta')$ is a minimal model;
        \item if $K_X+\Delta$ is not pseudoeffective, then $(X',\Delta')$ admits a $(K_{X'}+\Delta')$-Mori fiber space structure.
    \end{itemize}
\end{conjecture}

\begin{conjecture}\label{c-ABU}
    Let $S$ be either an integral smooth affine curve over a field of characteristic zero or the spectrum of a DVR of positive or mixed characteristic, and let $s\in S$ be a closed point. Let $\pi\colon (X,\Delta)\to S$ be a family of klt $TNR_1$ pairs. If $K_X+\Delta$ is nef then it is semiample.
\end{conjecture}

\begin{remark}
    We point out that, when $S$ is of positive or mixed characteristic, \autoref{c-ABU} holds if and only if Abundance holds for klt pairs over \textit{fields}. Indeed, suppose $\pi\colon (X,\Delta)\to S$ is a projective family of klt $TNR_1$ pairs. Then $K_{X_\eta}+\Delta_\eta$ and $K_{X^\nu_s}+\Delta_s^\nu$ are both semiample, in particular so is $K_{X_s}+\Delta_s$ by \autoref{l-keelfrobeniusdescent}. Then $K_X+\Delta$ is semiample over $S$ by \cite{CT} and \cite{Wit_RelSA}.
\end{remark}

\subsection{Boundedness}

\begin{definition}
    Let $k$ be any field and let $\cB=\lbrace Y_i\rbrace_{i\in I}$ be a class of projective $k$-varieties. We say $\cB$ is \textit{bounded} (resp. \textit{birationally bounded}) if there exists a surjective projective morphism of finite type $k$-schemes $\beta\colon \cY\to\cS$ such that for any $i\in I$ there exists $s\in\cS(k)$ and an isomorphism $Y_i\cong \cY_s$ (resp. a birational isomorphism $Y_i\simeq\cY_s$). Such a $\beta$ is called a \textit{bounding family for }$\cB$ (resp. a \textit{birationally bounding family for} $\cB$).
\end{definition}

\begin{definition}
        Let $k$ be any field and let $\cS$ be a $k$-scheme of finite type. A \textit{stratification of $\cS$} is a collection $\lbrace \cS_j\rbrace_{j\in J}$ of finitely many pairwise disjoint locally closed subschemes such that $\cS=\bigsqcup_{j\in J} \cS_j$ set-theoretically.
\end{definition}

\begin{remark}
    If $\beta\colon\cY\to \cS$ is a (birationally) bounding family for a class of projective varieties $\cB$ and $\cS'\coloneqq \bigsqcup_{j\in J}\cS_j$ is a stratification of $\cS$, then the base change
    \[
    \beta'\colon \cY\times_\cS\cS'\to\cS'
    \]
    via the natural morphism $\cS'\to \cS$ is again a (birationally) bounding family for $\cB$.
\end{remark}

\begin{lemma}\label{l-simnor}
Let $k$ be a perfect field, let $\cB=\lbrace Y_i\rbrace_{i\in I}$ be a class of projective $k$-varieties, and let $\cB^\nu=\lbrace Y^{\nu}_i\rbrace_{i\in I}$ be the class of the normalizations. If $\cB$ is bounded then so is $\cB^\nu$. Furthermore, we can pick a bounding family for $\cB^{\nu}$
\[
\bigsqcup_{j\in J}\beta_j^{\nu}\colon\bigsqcup_{j\in J}\cY^{\nu}_j\to\bigsqcup_{j\in J}\cS_j
\] 
where each $\beta_j^{\nu}$ is a flat projective surjective morphism with geometrically integral and normal fibers, such that $\cY_j$ normal and $\cS_j$ is smooth for all $j\in J$.
\end{lemma}

\begin{proof}
Let $\beta\colon \cY\to\cS$ be a bounding family for $\cB$. Without loss of generality we may assume that $\cS$ is reduced. Smoothness of $\cS$ is an open condition, and by \cite[Tag 0ASY]{SP} $\beta$ has a flattening stratification. Combining this with \cite[Theoreme 12.2.1]{EGAIV_3} and \cite[Theorem 12, Lemma 21]{KolSN} (see also \cite{CL}) we obtain that the scheme $\cS$ has a stratification $\cS'\coloneqq\bigsqcup_{j\in J}\cS_j$ such that, letting 
\[
\beta_j\colon \cY_j\coloneqq \cY\times_{\cS}\cS_j\to\cS_j
\]
be the morphism induced by base change, the following hold for all $j\in J$:
\begin{itemize}
    \item $\cS_j$ is smooth and $\cY_j$ is integral;
    \item $\beta_j$ is flat, surjective, and with integral geometric fibers;
    \item there exists a simultaneous normalization (\cite[Definition 11]{KolSN}) $\cY^\nu_j\to\cY_j$. 
\end{itemize}
Then 
\[
\bigsqcup_{j\in J}\beta_j^{\nu}\colon\bigsqcup_{j\in J}\cY^{\nu}_j\to\bigsqcup_{j\in J}\cS_j
\] 
is a bounding family for $\cB^{\nu}$.
\end{proof}

\begin{lemma}\label{l-simres}
Let $k$ be an algebraically closed field of characteristic $p>0$, let $\beta\colon\cY\to\cS$ be a flat projective surjective morphism of dimension three between $k$-schemes of finite type, where $\cS$ is integral, smooth, and $\cY$ and the geometric fibers of $\beta$ are integral and normal. Then there exists a dominant morphism of integral schemes $\cS'\to \cS$, which is finite onto its image, such that the induced morphism $\cY'\coloneqq \cY\times_\cS \cS'\to \cS'$ admits a simultaneous resolution.
\end{lemma}

\begin{proof}
Let $K$ be the function field of $\cS$ and let $Y\coloneqq \cY_{\overline{K}}$ be the geometric generic fiber of $\beta$. By \cite{CP19} there exists a resolution of singularities of $\overline{K}$-varieties
\begin{equation}\label{e-ros}
f\colon X=X_n\xrightarrow{f_{n-1}} X_{n-1}\xrightarrow{f_{n-2}}...\xrightarrow{f_0} X_0= Y
\end{equation}
where each $f_i$ is a blowup along a $\overline{K}$-smooth center $Z_i\subset X_i$. For every $i=0,...,n$ we can find finite extensions of $K$, $K_i\subset\overline{K}$, and $K_i$-schemes $Z'_i\subset X'_i$ such that $(Z'_{i,\overline{K}}\subset X'_{i,\overline{K}})=(Z_i\subset X_i)$. Note that $Z'_i$ is $K_i$-smooth by \cite[Tag 02YJ]{SP}. Let now $L$ be a finite extension of $K$ contained in $\overline{K}$ and containing $K_i$ for all $i$. By the usual ``spreading out'' technique (see for example the proof of \cite[Corollary 1.10]{DW}) we can find a dominant morphism of integral normal $k$-schemes $\cS'\to \cS$ which is finite onto its image, such that $L$ is the function field of $\cS'$, and a sequence of morphisms of integral normal projective and flat $\cS'$-schemes
\begin{equation}\label{e-ros'}
\tilde{f}\colon \cX=\cX_n\xrightarrow{\tilde{f}_{n-1}} \cX_{n-1}\xrightarrow{\tilde{f}_{n-2}}...\xrightarrow{\tilde{f}_0} \cX_0= \cY'.
\end{equation}
Here $\cY'\coloneqq \cY\times_\cS \cS'$, each $\tilde{f}_i$ is a blowup along a $\cS'$-smooth center $\cZ_i\subset \cX_i$, and \autoref{e-ros} is the base change of \autoref{e-ros'} through $\Spec (\overline{K})\to \cS'$. In particular $\cX\to\cS'$ is generically smooth. Hence, up to replacing $\cS$ and $\cS'$ with open subsets, we may assume it is actually smooth and the induced morphism $\tilde{f}_s\colon \cX_s\to \cY'_s$ is a resolution of singularities for all $s\in\cS'$.
\end{proof}


\section{Relative Iitaka fibrations}\label{s-relative_iitaka_fibration}

Let $B$ be an integral regular affine scheme, let $X$ be an integral normal proper $B$-scheme, let $D$ be an effective Cartier divisor, and consider the dominant rational maps of $B$-schemes
\[
\phi_{|mD|}\colon X\dashrightarrow Z_m\subset \bP_B^{h^0(X,mD)-1},
\]
where $Z_m$ denotes the image of $\phi_{|mD|}$. By \cite[Sections 2.1.A and 2.1.B]{La1} for all $m>0$ sufficiently divisible the maps $\phi_{|mD|}$ are birational to a fixed morphism of integral normal $B$-schemes
\[
\phi_{\infty}\colon X_{\infty}\to Z_{\infty},
\] called the \textit{Iitaka fibration of $D$ over $B$}, satisfying $\phi_{\infty,\ast}\cO_{X_{\infty}}=\cO_{Z_{\infty}}$. By a slight abuse of notation we will also refer to the maps $\phi_{|mD|}$ as the Iitaka fibration of $D$, whenever $m>0$ is sufficiently divisible (this shall not cause too much confusion, as we will primarily be concerned with semiample $D$). The above definitions extend naturally to $\bQ$-Cartier $\bQ$-divisors.

\begin{definition}\label{d-semiample}
		Let $\pi \colon X \to B$ be a proper morphism. A $\bQ$-Cartier $\bQ$-divisor $D$ on $X$ is said to be \textit{semiample over $B$} (or $\pi$-\textit{semiample}) if there exists $m>0$ such that $mD$ is integral and globally generated over $B$, that is the natural map $\pi^*\pi_*\cO_X(mD) \to \cO_X(mD)$ is surjective. 
	\end{definition} 
		
\begin{theorem}[{\cite[Theorem 2.1.26]{La1}}]\label{t-semiamplefibration}
		Let $X$ be an integral normal proper $B$-scheme and let $D$
    be a $\bQ$-Cartier $\bQ$-divisor on $X$. Then the following are equivalent.
\begin{itemize}
    \item $D$ is semiample over $B$;
    \item there is a contraction of $B$-schemes $f\colon X\to Z$ such that $f$ is the morphism induced
          by $|mD/B|$ for all sufficiently divisible $m$;
    \item There is a contraction of $B$-schemes $f \colon X \to Z$ such that $D\sim_{\bQ,B}f^*A$ where $A$ is a $\bQ$-divisor which is ample over $B$.
\end{itemize}
\end{theorem}

\begin{remark}[Section rings and ample models]
    The contraction $f$ is the same in points (b) and (c) of \autoref{t-semiamplefibration}, and it is the Iitaka fibration of $D$ (also called the \textit{semiample contraction of D}). Since $D$ is semiample, then the \textit{section ring} 
    \[
    R(X/B,D)\coloneqq \bigoplus_{m\geq 0}\pi_*\cO_X(\lfloor mD\rfloor)
    \]
    is a sheaf of finitely generated graded $\cO_B$-algebras, and $f$ is the natural morphism of $B$-schemes $f\colon X\to \relProj_B (R(X/B,D))$. More in general, whenever $R(X/B,D)$ is finitely generated, the Iitaka fibration of $D$ is given by the natural dominant rational map of $B$-schemes $f\colon X\dashrightarrow \relProj_B (R(X/B,D))$ (see \cite[Section 2.1]{La1}). In this case $X^\Amp\coloneqq \relProj_B (R(X/B,D))$ is the ample model of $D$ over $B$, introduced in \autoref{d-D-negative contraction_ample_model}. In particular, if $(X,\Delta)$ is a dlt pair over $B$ such that $K_X+\Delta$ is big and finitely generated, then its ample model coincides with the canonical model (\autoref{d-goodminimalcanonical_model}).
\end{remark}

\begin{lemma}\label{l-f_scontraction_fcontraction}
    Let $B$ be an integral regular affine scheme and let $X$ be an integral normal scheme. Let $\pi\colon X\to B$ be a flat contraction whose fibers are normal varieties, and let $f\colon X\to Y$ be a proper morphism of flat $B$-schemes. Then the following hold.
    \begin{itemize}
        \item[(1)] If $f_s$ is a contraction for some closed point $s\in B$ then $f$ and $f_u$ are contractions for all $u$ in an open neighborhood of $s$.
        \item[(2)] If $f$ is a contraction, then $f_s$ is a contraction for all $s\in B$ general.
    \end{itemize}
\end{lemma}

\begin{proof}
    We begin by proving (1). We first show that $f$ is a contraction. By the theorem on formal functions (\cite[Theorem III.11.1]{Har}) we can replace $X$ and $Y$ with their completion along $X_s$ and $Y_s$, which we denote by $\frX$ and $\frY$, respectively. We can express these formal schemes as limits of extension of the following kind
    \begin{center}
        \begin{tikzcd}
        X_0\arrow[d,"f_0"]\arrow[r,hook]&\cdots\arrow[r,hook]&X_{n}\arrow[d,"f_n"]\arrow[r,hook] & X_{n+1}\arrow[d,"f_{n+1}"]\arrow[r,hook]&\cdots\arrow[r,hook]&\frX\arrow[d,"\hat{f}"]\\
            Y_0\arrow[r,hook]\arrow[d]&\cdots\arrow[r,hook]&Y_{n}\arrow[d]\arrow[r,hook] & Y_{n+1}\arrow[d]\arrow[r,hook]&\cdots\arrow[r,hook] & \frY\arrow[d]\\
            \Spec (R_0)\arrow[r,hook]&\cdots\arrow[r,hook]&\Spec (R_{n})\arrow[r,hook] & \Spec (R_{n+1})\arrow[r,hook] & \cdots\arrow[r,hook] & \Spf (R), 
        \end{tikzcd}
    \end{center}
    where $R\coloneqq \widehat{\cO_{B,s}}$ and $\ker (R_{n+1}\to R_n)$ is a principal square-zero ideal which, as an $R_{n+1}$-module, is isomorphic to $R_{r}$ for some $r\leq n$. We will show $f_{n+1,*}\cO_{X_{n+1}}=\cO_{Y_{n+1}}$ by induction. Note that we have the following morphism of exact sequences of $\cO_{Y_{n+1}}$-modules
    \begin{center}
        \begin{tikzcd}
            0\arrow[r] & f_{r,*}\cO_{X_{r}}\arrow[r]           & f_{n+1,*}\cO_{X_{n+1}}\arrow[r]            & f_{n,*}\cO_{X_n}\arrow[r]             & R^1f_{r,*}\cO_{X_r}\\
            0\arrow[r] & \cO_{Y_r}\arrow[r]\arrow[u,"\cong"] & \cO_{Y_{n+1}}\arrow[r]\arrow[u,hook] & \cO_{Y_{n}}\arrow[r]\arrow[u,"\cong"] & 0.\arrow[u]
        \end{tikzcd}
    \end{center}
    The lower row is exact by the infinitesimal criterion for flatness (\cite[Tag 00MD]{SP}), while the first and third vertical arrows are isomorphism by induction.
    In particular, we see that the map $f_{n+1,*}\cO_{X_{n+1}}\to f_{n,*}\cO_{X_{n}}$ is surjective, so we actually have a morphism of short exact sequences. Since both $X_{n+1}$ and $Y_{n+1}$ are faithfully flat over $R_{n+1}$ and $f_{n+1}$ is surjective, we have that the central arrow is injective. A straightforward diagram chase shows then that it is an isomorphism.

    We now show that $f_u$ is a contraction for all $u$ in a neighborhood of $s$. We will denote by $\sigma\colon Y\to B$ the structure morphism. Let $\cO_Y(l)$ be a sufficiently ample line bundle on $Y$ and let $L\coloneqq f^*\cO_Y(l)$. Consider now the short exact sequence
    \[
    0\to L\otimes \pi^*\mathfrak{m}_u\to L\to L_u\to 0,
    \]
    where $\mathfrak{m}_u\subset \cO_B$ denotes the maximal ideal of $u$. By pushing forward via $f$ we see that $f_u$ is a contraction if and only if $f_*L\to f_{u,*}L_u$ is surjective or, equivalently, if and only if the map
    \begin{equation}\label{e-H^1map}
        R^1f_*L\otimes\sigma^*\mathfrak{m}_u\to R^1f_*L
    \end{equation}
    is injective. By the projection formula we have $R^1f_*L=R^1f_*\cO_X\otimes \cO_Y(l)$ so, upon enlarging $l$, we may assume that the left and right hand side of \autoref{e-H^1map} are generated by global sections, and that $H^i(Y_u,R^1f_{u,*}\cO_{X_u}\otimes \cO_{Y_u}(l))$ and $H^i(Y,R^1f_{*}\cO_{X}\otimes \cO_{Y}(l))$ both vanish for all $u\in B$ and all $i>0$. But then, injectivity of \autoref{e-H^1map} can be checked on global sections, i.e. it is equivalent to the injectivity of 
    \begin{equation}\label{e-H^1inj1}
        R^1\pi_*L\otimes \mathfrak{m}_u \xrightarrow{t_u} R^1\pi_*L.
    \end{equation}
    Since $f_s$ is a contraction then $t_s$ is injective, i.e. $R^1\pi_*L$ is flat around $s$ (\cite[Tag 00MK]{SP}). But flatness is open, hence $t_u$ is injective for all $u$ in a neighborhood of $s$. This concludes the proof of item (1). As for item (2), observe that the argument above tells us that the locus $\lbrace s\in B \textup{ such that }f_s\textup{ is a contraction}\rbrace$ coincides with the locus on $B$ where $R^1\pi_*L$ is flat, and the latter is a dense open (\cite[Tag 0ASY]{SP}).
\end{proof}

\begin{lemma}\label{l-stein_invariance+}
    Let $S$ be either an integral smooth affine curve over a field of characteristic zero or the spectrum of a DVR of positive or mixed characteristic, and let $s\in S$ be a closed point. Let $X$ be an integral normal scheme, let $\pi\colon X\to S$ be a contraction whose fibers are normal varieties, and let $D$ be a semiample Cartier divisor on $X$. Then the following are equivalent.
        \begin{itemize}
            \item[(i)] The restriction map $H^0(X,mD)\to H^0(X_s,mD_s)$ is surjective for all $m\geq 0$ divisible enough;
            \item[(ii)] $f_s$ is a contraction.
        \end{itemize}
        If furthermore $\bQ\subset\cO_S$, or $D_s$ is big, then conditions (i) and (ii) are equivalent to $Y_s$ being normal.
\end{lemma}

\begin{proof}
    We have a commutative diagram
    \begin{equation*}
        \begin{tikzcd}
            X\arrow[rr,"f"]   &           & Y\\
            X_s\arrow[rr,"f_s"]\arrow[u,"i",hook]\arrow[dr,"\bar{f}_s"] &           & Y_s\arrow[u,"j",hook]\\
                & \bar{Y}_s,\arrow[ur,"h_s"] & 
        \end{tikzcd}    
    \end{equation*}
    where the square is Cartesian and the triangle is the Stein factorization of $f_s$. Let $m_0$ be a positive integer such that $m_0D$ and $m_0D_s$ induce the respective semiample contraction, let $m\in m_0\bN$, so that $mD\sim f^*A$ for some very ample divisor $A$ on $Y$, and let $\bar{A}_s\coloneqq h^*_sA_s$. We have induced pullback and restriction maps
    \begin{equation}\label{e-restriction&pullback}
        \begin{tikzcd}
            H^0(X,mD)\arrow[d] & & H^0(Y,A)\arrow[d]\arrow[ll,"f^*",swap]\\
            H^0(X_s,mD_s) & H^0(\bar{Y}_s,\bar{A}_s)\arrow[l,"\bar{f}_s^*",swap] & H^0(Y_s,A_s),\arrow[l,"h^*_s",hook',swap]
        \end{tikzcd}
    \end{equation}
    where $f^*$ and $\bar{f}_s^*$ are isomorphisms by the projection formula. It follows then that $\bar{f}_s$ is the morphism induced by $H^0(X_s,mD_s)$. As $\bar{f}_s$ is a contraction by definition, it is the semiample contraction of $D_s$. As $A$ is ample, upon enlarging $m_0$ we may assume that the rightmost vertical map in \autoref{e-restriction&pullback} is surjective. Hence we see that the leftmost vertical map is a surjection if and only if $h_s^*$ is an isomorphism. We conclude by \autoref{p-AmpleUH}.

    As for the ``furthermore'' part, if $f_s$ is a contraction then clearly $Y_s$ is normal. The converse follows from \autoref{l-finitepartSFconnfibisUH}.
\end{proof}

\begin{theorem}\label{t-normalimageK_X+Dnonpositivefibration}
    Let $C$ be an integral smooth affine curve over a field of characteristic zero. Let $\pi\colon (X,\Delta)\to C$ be a projective family of klt pairs. Let $f\colon X\to Y$ be a contraction of $C$-schemes such that either
    \begin{enumerate}
        \item $K_X+\Delta$ is semiample and $f$ is the associated contraction; or
        \item $K_X+\Delta$ is $f$-antiample.
    \end{enumerate}
    Then $Y_c$ is normal for all closed points $c\in C$.
\end{theorem}

\begin{proof}
Consider case (b) first. Fix $c\in C$ a closed point and let $H$ denote an ample divisor on $Y$. If $H$ is sufficiently positive, then $-K_X-\Delta+f^*H$ is ample. In particular, by Bertini theorem (\cite[Theorem 4.8]{KolSOP}) we can find a general element $\Theta\in |-K_X-\Delta+f^*H|_{\bQ}$ so that, after possibly shrinking $C$ around $c$ we have that $\pi\colon (X,\Delta+\Theta)\to C$ is still a family of klt pairs, $K_X+\Delta+\Theta$ is semiample, and $f$ is the associated contraction. Hence we have reduced (b) to (a).

Let us consider case (a) now. The pairs $(X,\Delta)$ and $(X,\Delta+X_c)$ induce generalized polarized pairs $(Z,M_Z+B_Z)$ and $(Z,M_Z+B_Z+Z_c)$ with glc singularities by \cite[Lemma 6.8]{Birlcfib}. We have the following commutative diagram
\[
		\xymatrix{
			X' \ar[d]_{f'}   \ar[r]^{\alpha} &  X \ar[d]^{f}  \\
			Z' \ar[r]^{\beta} &  Z
		}
		\]
where both $\alpha$ and $\beta$ are projective birational morphisms, with $X'$ and $Z'$ smooth. If we define $K_{X'}+\Delta'+X_c'$ by crepant pullback of $K_X+\Delta+X_c$, this sublc pair induces another generalized polarized pair structure $(Z',M_{Z'}+B_{Z'}+Z_c')$ with subglc singularities and such that $M_Z$ is nef. We have the ramification formula
$$K_{Z'}+M_{Z'}+B_{Z'}+Z_c'=\beta^*(K_Z+M_Z+B_Z+Z_c).$$
Note that, unlike $B_Z$ and $Z_c$, it may be the case that $B_{Z'}$ and $Z_c'$ have common components.

We claim that $(Z',B_{Z'}+Z'_c)$ is subplt. Note that we can assume that $(Z',B_{Z'}+Z'_c)$ and $(X',\Delta'+X_c')$ are log smooth. By \cite[Corollary 2.31]{KM} it is then enough to show that if $E$ is a prime divisor on $Z'$ such that $E\neq \beta_*^{-1}Z_c$, then $\coeff_E(B_{Z'}+Z_c')<1$, i.e. $\lct_E(X',\Delta'+X_c';f'^*E)>0$. There are now two possibilities: if $E$ is horizontal over $C$, then  $\lct_E(X',\Delta'+X_c';f'^*E)=\lct_E(X',\Delta';f'^*E)>0$, where the latter inequality is a consequence of $(X',\Delta')$ being subklt. If $E$ is vertical over $C$, then we may assume it is supported over $c$, after possibly shrinking $C$. As $Z_c$ is irreducible and $E\neq \beta_*^{-1}Z_c$ we then have that $E$ is $\beta$-exceptional. Hence $f'^*E$ is a sum of $\alpha$-exceptional components of $X_c'$, in particular $f'^*E$ does not contain $\alpha_*^{-1}X_c$ in its support. As $(X',\Delta'+X_c')$ is subplt, we then have $\lct_E(X',\Delta'+X_c';f'^*E)>0$, thus the claim is proven. 

Let $A$ and $G$ be $\bQ$-divisors, such that $A$ is ample and $G\geq 0$ and $K_{Z'}+M_{Z'}+B_{Z'}+Z'_c\sim_{\bQ}A+G$. Let $\epsilon$ be a very small positive rational number. After possibly replacing $G$ by $G-\coeff_{\beta_*^{-1} Z_c}(G)Z_c'$, we may assume that $G$ does not contain $\beta_*^{-1}Z_c$ in its support, hence $K_{Z'}+B_{Z'}+Z'_c+\epsilon G$ is still subplt. Further, as $M_{Z'}$ is nef, we have $M_{Z'}+\epsilon A$ is $\bQ$-linearly equivalent to $\delta H$, where $H\geq 0$ is a general very ample divisor and $0<\delta\ll 1$. Thus we get a obtain pair $K_{Z'}+B_{Z'}+Z_c'+\epsilon G+\delta H$
which is still subplt and such that $\beta_*^{-1}Z_c$ is its only plt center. Note that all the components of $B_{Z'}+Z_c'+\epsilon G$ with negative coefficient are exceptional over $Z$. As
$$K_{Z'}+B_{Z'}+Z_c'+\epsilon G+\delta H\sim_{\bQ}(1+\epsilon)(K_{Z'}+M_{Z'}+B_{Z'}+Z_c')$$
we have that $\beta$ is a $(K_{Z'}+B_{Z'}+Z_c'+\epsilon G+\delta H)$-trivial contraction, hence $K_Z+\beta_*(B_{Z'}+Z_c'+\epsilon G+\delta H)$ is plt by \cite[Lemma 3.38]{KM}. In particular, $Z_c$ is a plt center, hence it is normal by \autoref{l-pltcentersnormalUH}.
\end{proof}

\begin{theorem}\label{t-formallift_(X,L)_Lsa}
    Let $X_0$ be a proper normal $k$-scheme, let $L_0$ be a semiample line bundle on $X_0$, and let $f_0\colon X_0\to Y_0$ be its semiample contraction. Suppose $R^1f_{0,*}\cO_{X_0}=0$. Then there exists $m_0\geq 1$ such that, for any formal lifting $(\frX,\frL)\to\Spf (R)$ of $(X_0,L_0)$, the restriction map
    \[
    H^0(\frX,\frL^m)\to H^0(X_0,L_0^m)
    \]
    is surjective for all $m\in m_0\bN$.
\end{theorem}

\begin{proof}
We can express our formal scheme as the limit of a series of Cartesian extensions
\begin{center}
        \begin{tikzcd}
        X_0\arrow[d]\arrow[r,hook]&\cdots\arrow[r,hook]&X_{n}\arrow[d]\arrow[r,hook] & X_{n+1}\arrow[d]\arrow[r,hook]&\cdots\arrow[r,hook]&\frX\arrow[d]\\
        \Spec (R_0)\arrow[r,hook]&\cdots\arrow[r,hook]&\Spec (R_{n})\arrow[r,hook] & \Spec (R_{n+1})\arrow[r,hook] & \cdots\arrow[r,hook] & \Spf (R) 
        \end{tikzcd}
\end{center}
where $\ker(R_{n+1}\to R_n)$ is a square-zero principal ideal which, as an $R_{n+1}$-module, is isomorphic to $R_{r}$ for some $r\leq n$. Let $m_0\geq 1$ be such that $f_0$ is the map induced by $|L_0^{m_0}|$, and let $A_0$ be a very ample line bundle on $Y_0$ such that $L_0^{m_0}\cong f_0^*A_0$. Upon replacing $m_0$ with a multiple we may also assume $H^j(Y_0,A_0^l)=0$ for all $j,l\geq 1$. 

Let ($\mathrm{P}_n$) be the following property:
\begin{itemize}
    \item[($a_n$)] the restriction map $H^0(X_i,L_i^m)\to H^0(X_{i-1},L_{i-1}^m)$ is surjective for all $i\leq n$ and all $m\in m_0\bN$; and
    \item[($b_n$)] there exist Cartesian diagrams
    \begin{center}
        \begin{tikzcd}
        X_0\arrow[d,"f_0"]\arrow[r,hook] & X_1\arrow[d,"f_1"]\arrow[r,hook] & \cdots\arrow[r,hook] & X_n\arrow[d,"f_n"]\\
        Y_0\arrow[r,hook]\arrow[d] & Y_1\arrow[r,hook]\arrow[d] & \cdots\arrow[r,hook] & Y_n\arrow[d]\\
        \Spec (R_0)\arrow[r,hook] & \Spec (R_1)\arrow[r,hook] & \cdots\arrow[r,hook] & \Spec (R_n)\\
        \end{tikzcd}
    \end{center}
    where the $Y_i$ are flat and proper $R_i$-schemes with a very ample line bundle $A_i$ such that $A_i|_{Y_{i-1}}=A_{i-1}$, $f_i^*A_i\cong L_i^{m_0}$ , and the $f_i$ are contractions such that $R^1f_{i,*}\cO_{X_i}=0$ for all $i\leq n$.
\end{itemize}
We will show by induction that ($\mathrm{P}_n$) holds for all $n\geq 0$. 
Note that ($\mathrm{P}_0$) clearly holds. Assume now ($\mathrm{P}_n$) and consider the short exact sequence
\[
0\to L_{r}^m\to L_{n+1}^m\to L_n^m\to 0.
\]
To show ($a_{n+1}$) it suffices to show the vanishing $H^1(X_{r},L_{r}^m)=0$. Writing $m=lm_0$ for some $l\geq 1$, by the Grothendieck spectral sequence and projection formula we have an exact sequence 
\[
0\to H^1(Y_{r},A^l_{r}) \to H^1(X_{r},L_{r}^m)\to H^0(Y_{r}, R^1f_{r,*}\cO_{X_{r}}\otimes A^l_{r}).
\]
The rightmost term vanishes, as $R^1f_{r,*}\cO_{X_r}=0$ by induction. As for the leftmost one, note that by base change we have an injective map
\[
H^1(Y_{r},A_{r}^l)\otimes_R R_0\hookrightarrow H^1(Y_0,A_0^l)=0.
\]
Hence ($a_{n+1}$) holds. To show ($b_{n+1}$), observe that by ($a_{n+1}$) the line bundle $L^m_{n+1}$ induces a morphism $f_{n+1}\colon X_{n+1}\to Y_{n+1}$ extending $f_{n}$, and that we have a very ample line bundle $A_{n+1}$ such that $f_{n+1}^*A_{n+1}\cong L_{n+1}^m$. Also, note that $f_{n+1}$ is a contraction by \autoref{l-f_scontraction_fcontraction}. By taking $Rf_{n+1,*}$ of the short exact sequence
\[
0\to \cO_{X_{r}}\to \cO_{X_{n+1}}\to \cO_{X_n}\to 0
\]
we get the vanishing $R^1f_{n+1,*}\cO_{X_{n+1}}=0$ and a short exact sequence
\[
0\to \cO_{Y_{r}}\to \cO_{Y_{n+1}}\to \cO_{Y_n}\to 0.
\]
In particular, $Y_{n+1}$ is flat over $R_{n+1}$ (\cite[Tag 00MD]{SP}).
\end{proof}


\section{MMP in families of \texorpdfstring{$TNR_1$}{TNR1} pairs} \label{s-MMPfam}

\begin{definition}
    Let $\pi\colon X \to B$ be a projective morphism and let $D$ be a $\bQ$-Cartier $\bQ$-divisor on $X$. The \emph{diminished stable base locus of $D$ over $B$}  is
    \[
    \mathbf{B}_{-}(D/B) = \bigcup_{A \ \pi\text{-ample }\bQ\text{-divisor}} \mathbf{B}(D+A/B),
    \]
    where $\mathbf{B}(D+A/B)$ is the stable base locus of $D+A$ over $B$. If $B$ is clear from the context, we will simply write $\mathbf{B}_{-}(D)$.
\end{definition}

\begin{lemma}[{\cite[Lemma 2.26]{BBS}}]\label{l-dimbaselocus_and_MMP}
    Let $(X,\Delta)$ be a projective klt pair over a regular integral affine scheme $B$. Let $f\colon X\dashrightarrow Y$ be a birational step of the $(K_X+\Delta)$-MMP over $B$, let $\Gamma\coloneqq f_*\Delta$, and let 
    \begin{equation*}
        \begin{tikzcd}
                                            & W\arrow[dl,"\alpha",swap]\arrow[dr,"\beta"] & \\
            X \arrow[rr,dashed,"f"]         &                                             & Y
        \end{tikzcd}        
    \end{equation*}
    be a resolution of indeterminacies. Then $\beta^{-1}\mathbf{B}_{-}(K_Y+\Gamma) \subset \alpha^{-1}\mathbf{B}_{-}(K_X + \Delta).$
\end{lemma}

\begin{theorem}\label{t-MMP-topfamilies}
   Let $S$ be either an integral smooth affine curve over a field of characteristic zero or the spectrum of a DVR of positive or mixed characteristic, and let $s\in S$ be a closed point. Let $\pi_X\colon (X,\Delta)\to S$ be a projective family of $TNR_1$ pairs. Assume
   \begin{itemize}
       \item[(I)] the pair $(X,\Delta+X_s)$ is plt;
       \item[(II)] if $V\subset\mathbf{B}_-(K_X+\Delta)$ is a non-canonical center of $(X,\Delta+X_s)$, then $V$ is $\pi_X$-horizontal.
   \end{itemize}
   Let $f\colon X\dashrightarrow Y$ be a step of the $(K_X+\Delta)$-MMP over $S$. Assume either
   \begin{itemize}
       \item[(a)] $S$ is of equicharacteristic zero; or 
       \item[(b)] $S$ is of positive or mixed characteristic.
   \end{itemize}

   If $f$ is a $(K_X+\Delta)$-Mori fiber space then, respectively,
   \begin{itemize}
       \item[(a)] $f_s$ is a $(K_{X_s}+\Delta_s)$-Mori fiber space as well;
       \item[(b)] there is a commutative diagram
       \begin{equation*}
           \begin{tikzcd}
            X_s^\nu\arrow[d,"\nu"]\arrow[dr,"f_s^\nu"]\arrow[r,"\bar{f_s}"] & \bar{Y_s}\arrow[d,"h_s"]\\
               X_s\arrow[r,"f_s"]     & Y_s
           \end{tikzcd}
       \end{equation*}
       where $f_s^\nu=h_s\circ \bar{f_s}$ is the Stein factorization, $\bar{f_s}$ is a $(K_{X_s^\nu}+\Delta_s^\nu)$-Mori fiber space, and $h_s$ is a finite universal homeomorphism.
   \end{itemize}

   If $f$ is a $(K_X+\Delta)$-negative birational contraction then, respectively,
   \begin{itemize}
       \item[(a)] $f_s$ is a $(K_{X_s}+\Delta_s)$-negative birational contraction;
       \item[(b)] there is a commutative diagram
       \begin{equation*}
           \begin{tikzcd}
            X_s^\nu\arrow[d,"\nu_X"]\arrow[r,dashed,"\bar{f}_s"] & Y_s^\nu\arrow[d,"\nu_Y"]\\
               X_s\arrow[r,"f_s",dashed]     & Y_s
           \end{tikzcd}
       \end{equation*}
       where the normalizations $\nu_X$ and $\nu_Y$ are both small universal homeomorphisms, the natural map $\bar{f}_s$ is a $(K_{X_s^\nu}+\Delta_s^\nu)$-negative birational contraction.
   \end{itemize}
   In particular, when $f$ is birational, $\pi_Y\colon (Y,\Gamma\coloneqq f_*\Delta)\to S$ is again a projective family of $TNR_1$ pairs such that (I) and (II) hold. 
\end{theorem}

\begin{remark}\label{r-famtoppairs_char0_fampairs}
    Note that if $\pi_X\colon (X,\Delta)\to S$ is a projective family of $TNR_1$ pairs as in \autoref{t-MMP-topfamilies} and $\bQ\subset  \cO_S$, then $X_s$ is normal since it is a plt center (\autoref{l-pltcentersnormalUH}).
\end{remark}

\begin{proof}
    Suppose $f$ is a Mori fiber space. Then $f_s$ is certainly a morphism with positive-dimensional general fiber by upper-semicontinuity of fiber dimension, and all the fibers will be connected, since $f$ is a contraction. Note furthermore that $K_{X_s}+\Delta_s=(K_X+\Delta)|_{X_s}$ is $f_s$-antiample by construction. 
    \begin{itemize}
        \item[(a)] To conclude it is enough to show that $f_s\colon X_s\to Y_s$ is a contraction. By \autoref{l-finitepartSFconnfibisUH}, this is equivalent to $Y_s$ being normal, hence we conclude by \autoref{t-normalimageK_X+Dnonpositivefibration}.
        \item[(b)] Since $\nu$ is a universal homeomorphism, the induced morphism $f_s^\nu$ also has connected fibers. Hence the existence of the diagram follows by \autoref{l-finitepartSFconnfibisUH}. As $K_{X_s^\nu}+\Delta_s^\nu=\nu^*(K_{X_s}+\Delta_s)$ and $h_s$ is finite, we have that $K_{X_s^\nu}+\Delta_s^\nu$ is $\bar{f}_s$-antiample.
    \end{itemize}
    
    Suppose now that $f$ is birational. As $f$ does not extract divisors, we have that $Y_s$ is irreducible. By rigidity (\cite[Lemma 1.6]{KM}) we also have that $Y_s$ cannot be contracted, hence $f_s$ is birational. Furthermore, as $(X,\Delta+X_s)$ is plt, so will be $(Y,\Gamma+Y_s)$ by \autoref{l-negcontimprovesings}.

    Suppose $f$ is a divisorial contraction. It follows that $f_s$ is a proper birational morphism with connected fibers.
    \begin{itemize}
        \item[(a)] Again, it is enough to show that $f_s\colon X_s\to Y_s$ is a contraction. By \autoref{l-finitepartSFconnfibisUH}, this is equivalent to $Y_s$ being normal, hence we conclude by \autoref{t-normalimageK_X+Dnonpositivefibration}.
        \item[(b)] In this case we have a commutative diagram
        \begin{equation*}
           \begin{tikzcd}
            X_s^\nu\arrow[d,"\nu_X"]\arrow[dr,"f_s^\nu"]\arrow[r,"\bar{f_s}"] & Y_s^\nu\arrow[d,"\nu_Y"]\\
               X_s\arrow[r,"f_s"]     & Y_s,
           \end{tikzcd}
       \end{equation*}
       where $f_s^\nu=\nu_Y\circ \bar{f}_s$ is the Stein factorization. Note that $\nu_Y$ is a small universal homeomorphism, since $Y_s$ is a plt center (\autoref{l-pltcentersnormalUH}). As $K_{X_s^\nu}+\Delta_s^\nu=\nu_X^*(K_{X_s}+\Delta_s)$ and the latter is $f_s$-antiample, we have that $K_{X_s^\nu}+\Delta_s^\nu$ is $\bar{f}_s$-antiample as well. In particular, $\bar{f}_s$ is a $(K_{X_s^\nu}+\Delta_s^\nu)$-negative birational contraction by the negativity lemma (\cite[Lemma 1.17]{Kol_SMMP}).
    \end{itemize}

    Suppose now that $f$ is a flip. In both cases (a) and (b) we have a commutative diagram
    \begin{equation*}
           \begin{tikzcd}
            X_s^\nu\arrow[d,"\nu_X"]\arrow[r,dashed,"\bar{f}_s"] & Y_s^\nu\arrow[d,"\nu_Y"]\\
               X_s\arrow[r,"f_s",dashed]     & Y_s,
           \end{tikzcd}
       \end{equation*}
       where $\nu_Y$ is again a small universal homeomorphism as $Y_s$ is a plt center (both $\nu_X$ and $\nu_Y$ will be isomorphisms in case (a)). Our goal is to show that $f_s$ does not extract divisors. By contradiction, suppose this is not the case and let $D\subset Y_s$ be a prime divisor such that $N\coloneqq \textup{center}_{X_s}(D)$ has codimension $\geq 2$ in $X_s$. As $f$ is not an isomorphism at the generic point of $N$ we have $N\subset \mathbf{B}_-(K_X+\Delta)$. Furthermore \autoref{l-negcontimprovesings} yields
       \[
       a(D;X_s,\Delta_s)< a(D;Y_s,\Gamma_s)\leq 0,
       \]
       i.e. $N$ is a non-canonical center of $(X_s,\Delta_s)$. But now easy adjunction (\cite[Theorem 17.2]{FA}) yields
       \[
       0>\textup{totaldiscrep}(N;X_s,\Delta_s)\geq \textup{discrep}(N;X,\Delta+X_s).
       \]
       So $N$ is a $\pi_X$-vertical non-canonical center of $(X,\Delta+X_s)$ which is contained in $\mathbf{B}_-(K_X+\Delta/S)$, contradicting (II).

       By \autoref{l-negcontimprovesings} we have that $\pi_Y\colon (Y,\Gamma)\to S$ is again a projective family of $TNR_1$ pairs such that $(Y,\Gamma+Y_s)$ is plt. Suppose now that $V$ is a non-canonical center of $(Y,\Gamma+Y_s)$ and take a model $Z$ dominating $X$ and $Y$, and containing an exceptional divisor $E$ such that $V=\textup{center}_Y(E)$ and $a(E,Y,\Gamma+Y_s)<0$. Then by \autoref{l-negcontimprovesings} we must have $a(E,X,\Delta+X_s)\leq a(E,Y,\Gamma+Y_s) < 0$, hence the image $W$ of $E$ on $X$ is a non-canonical center of $(X,\Delta+X_s)$. By \autoref{l-dimbaselocus_and_MMP} if $V\subset \bB_-(K_Y + \Gamma)$ then we have $W \subset \bB_-(K_X + \Delta)$ as well. In which case $W$ is horizontal and hence so is $V$, therefore (II) holds as well.
\end{proof}

\begin{remark}[Flatness and higher dimensional bases]
    If $\bQ\subset \cO_S$ and $\pi_X\colon (X,\Delta)\to S$ is a projective family of pairs as in \autoref{t-MMP-topfamilies} the arguments in \autoref{t-MMP-topfamilies} extend to the case where $S$ is a smooth regular affine scheme of dimension $>1$ (see \cite[Lemmas 3.1 and 3.2]{HMX}). A somewhat delicate point is enforcing the flatness of $\pi_Y$. In characteristic zero one can argue as follows (we assume $\Delta=0$ for simplicity): let $\pi\colon X\to B$ be a projective family of normal $\bQ$-Gorenstein varieties with klt singularities. In particular $X$ itself has klt singularities. Suppose $f\colon X\dashrightarrow Y$ is a step of the $K_X$-MMP over $B$. If $f$ is birational then $Y$ is also klt, in particular Cohen-Macaulay. If $f$ is a Mori fiber space, then a canonical bundle formula argument analogous to the one in \autoref{t-normalimageK_X+Dnonpositivefibration} can be used to construct a boundary $\Delta_Y$ on $Y$ such that $(Y,\Delta_Y)$ is also klt hence, again, Cohen-Macaulay. The argument of \autoref{t-MMP-topfamilies} shows that $\pi_Y\colon Y\to B$ is equidimensional and $B$ is regular, we conclude in both cases by ``miracle flatness'' (\cite[Tag 00R4]{SP}).

    In positive characteristic this kind of argument breaks: for one, there is no canonical bundle formula. More importantly, there are klt singularities which are not Cohen-Macaulay (\cites{Tot_failKVfano, Tot_terminalnonCM3fold, Yas}). In this paper we will not address the higher-dimensional base case of \autoref{t-MMP-topfamilies}. However, it should be noted that even in situations where flatness is automatic (i.e. $S$ regular integral and of dimension one), relative MMP can fail to restrict to fiberwise MMP, even in families with terminal singularities (\cite{Bri_NEMMP}). For our purposes, the following easy result will suffice.
\end{remark}

\begin{proposition}\label{p-relativeMMPgeneralfiber}
    Let $k$ be a field, and let $B$ be an integral regular affine $k$-scheme. Let $\pi\colon X\to B$ be a projective family of smooth varieties. Let $\varrho\colon X\dashrightarrow Y$ be a $K_X$-minimal model over $B$. Then for a general closed point $s\in B$, the map $\varrho_s\colon X_s\dashrightarrow Y_s$ is a $K_{X_s}$-minimal model.
\end{proposition}

\begin{proof}
    Consider the following commutative diagram
    \begin{center}
        \begin{tikzcd}
             & W\arrow[dl,"p",swap]\arrow[dr,"q"] & \\
            X\arrow[rr,dashed,"\varrho"]\arrow[dr,"\pi",swap] &  & Y\arrow[dl]\\
             & B, & 
        \end{tikzcd}
    \end{center}
    where $W$ is the normalized closure of the graph of $\varrho$ inside $X\times_B Y$. As $\varrho$ is a $K_X$-negative birational contraction, we have
    \[
    p^*K_X=q^*K_Y+E
    \]
    for some effective $q$-exceptional divisor $E$ containing $q^{-1}_*\Exc(\varrho)$. Modulo shrinking $B$, we may assume
    \begin{itemize}
        \item $W\to B$ is flat and $W_s$ is the normalized closure of the graph of $\varrho_s$ inside $X_s\times Y_s$ for all $s\in B$;
        \item $E_s$ is an effective $q_s$-exceptional divisor, containing $q^{-1}_{s,*}\Exc(\varrho_s)$ for all $s\in B$.
    \end{itemize}
    As $q$ is a contraction, by point (2) of \autoref{l-f_scontraction_fcontraction} so is $q_s$ for all $s$ in a dense open. In particular, for all such $s$, the fiber $Y_s$ is normal. As $K_Y$ is $\bQ$-Cartier and nef, so is $K_{Y_s}$. Hence $\varrho_s$ is a $K_{X_s}$-minimal model for all such $s$. 
\end{proof}

\begin{theorem}\label{t-relcanstratification}
Let $k$ be an algebraically closed field of characteristic $p>5$. Let $B$ be a smooth affine $k$-variety and let $\pi\colon X\to B$ be a projective family of smooth general type 3-folds. Then, upon replacing $B$ by a dense open subset, the following hold:
\begin{itemize}
    \item[(1)] the canonical ring $R(X/B,K_X)=\bigoplus_{m\geq 0}H^0(X,mK_X)$ is a finitely generated graded $\cO_B$-algebra;
    \item[(2)] letting $f\colon X\dashrightarrow X^\Can\coloneqq \relProj_B (R(X/B,K_X))$ be the canonical model of $X$ over $B$, then $f_s\colon X_s\dashrightarrow (X^\Can)_s=(X_s)^\Can$ is the canonical model of $X_s$ for all $s\in B$ closed.
\end{itemize} 
\end{theorem}

\begin{proof}
By \cite[Corollary 1.10]{DW}, after replacing $B$ by a dense open, we can run a $K_X$-MMP over $B$, say $\varrho\colon X\dashrightarrow X^{\m}$, which terminates with a good minimal model. In particular the canonical ring $R(X^{\m}/B,K_{X^{\m}})$ is finitely generated since $K_{X^{\m}}$ is semiample. Since $\varrho$ is a $K_X$-negative birational contraction we have that $K_{X^{\m}}$ and $K_X$ have the same canonical ring, modulo passing to a Veronese subring (\autoref{r-non_positive_contractions_preserve_sections}). In particular, (1) holds.

Modulo further shrinking $B$, we may assume that $\varrho_s\colon X_s\dashrightarrow (X^\m)_s=(X_s)^\m$ is a good minimal model for all $s\in B$, by \autoref{p-relativeMMPgeneralfiber}. Let $g\colon X^\m\to X^\Can$ be the semiample contraction of $K_{X^\m}$ over $B$ (i.e. the canonical model of $K_{X^\m}$). By \cite[Corollary 1.3]{ABL} the canonical model $(X_s)^\Can$ has rational singularities, hence by \autoref{t-extension_R^1_vanishing}(b)\footnote{Note that the proof of \autoref{t-extension_R^1_vanishing}(b) depends only on \autoref{t-formallift_(X,L)_Lsa}} $g_s\colon X^\m_s\to (X^\Can)_s=(X_s)^\Can$ is the canonical model of $K_{X^\m_s}$ for all $s\in B$. As we have a factorization
\begin{center}
    \begin{tikzcd}
        f\colon X\arrow[r,"\varrho",dashed] & X^\m\arrow[r,"g"] & X^\Can,
    \end{tikzcd}
\end{center}
we conclude.

\end{proof}


\section{Proofs}\label{s-proofs}


\begin{proof}[Proof of \autoref{t-MMP+ABU=>AIP/p (charfree)}]

Note that $(X,\Delta+X_s)$ is plt by \cite[Theorem 4.9]{Kol_SMMP}. We can run a $(K_X+\Delta)$-MMP over $S$, which will terminate with either a Mori fiber space or a good minimal model. More precisely, by \autoref{t-MMP-topfamilies}, we have a commutative diagram
\begin{equation*}
    \begin{tikzcd}
        (X,\Delta)\arrow[r,"\varrho",dashed]   & (Y,\Gamma)\arrow[r,"f"]       & Z\\
                                  & (Y_s,\Gamma_s)\arrow[r,"f_s"]\arrow[u,hook]                  & Z_s\arrow[u,hook]\\
        (X_s,\Delta_s)\arrow[uu,hook]\arrow[ru,"\varrho_s",dashed]\arrow[r,"\bar{\varrho}_s",dashed]                        & (Y_s^\nu,\Gamma_s^\nu)\arrow[r,"\bar{f}_s"]\arrow[u,"\nu"]              & \bar{Z}_s,\arrow[u,"h_s"]
    \end{tikzcd}
\end{equation*}
where $\Gamma\coloneqq \varrho_*\Delta$ and
\begin{itemize}
    \item[(1)] $\varrho$ is a (birational) $(K_X+\Delta)$-MMP over $S$,
    \item[(2)] $\bar{\varrho}_s$ is a (birational) $(K_{X_s}+\Delta_s)$-MMP,
    \item[(3)] if $K_X+\Delta$ is not pseudoeffective then $f$, resp. $\bar{f}_s$, is a $(K_Y+\Gamma)$-, resp. $(K_{Y_s^\nu}+\Gamma_s^\nu)$-, Mori fiber space,
    \item[(4)] if $K_X+\Delta$ is pseudoeffective then $f$ and $\bar{f}_s$ are the semiample contractions of $K_Y+\Gamma$ and $K_{Y_s^\nu}+\Gamma_s^\nu$, respectively, and
    \item[(5)] in both cases $h_s$ is a universal homeomorphism by \autoref{l-finitepartSFconnfibisUH}.
\end{itemize}
By (2) and \autoref{r-non_positive_contractions_preserve_sections} we have $H^0(X_s,m(K_{X_s}+\Delta_s))\cong H^0(Y_s^\nu,m(K_{Y_s^\nu}+\Gamma_s^\nu))$ for all $m\geq 0$ divisible enough. If $K_X+\Delta$ is not pseudoeffective, then $H^0(Y_s^\nu,m(K_{Y_s^\nu}+\Gamma_s^\nu))=0$ for all $m\geq 1$, hence we conclude by upper-semicontinuity of cohomology.

Suppose now that $K_X+\Delta$ is pseudeffective. 

We have an induced commutative diagram of pullback and restriction maps.
\begin{equation*}
    \begin{tikzcd}
        H^0(X,m(K_X+\Delta))\arrow[dd,"\res_{X_s}"] & H^0(Y,m(K_Y+\Gamma))\arrow[l,"\varrho^*",hook', two heads]\arrow[d,"\res_{Y_s}"]       & H^0(Z,mA)\arrow[l,"f^*",hook', two heads]\arrow[d,twoheadrightarrow,"\res_{Z_s}"]\\
                                                    & H^0(Y_s,m(K_{Y_s}+\Gamma_s))\arrow[dl,"\varrho_s^*",swap]\arrow[d,"\nu^*",hook']                  & H^0(Z_s,mA_s)\arrow[l,"f_s^*",hook']\arrow[d,"h_s^*",hook']\\
        H^0(X_s,m(K_{X_s}+\Delta_s)                 & H^0(Y^\nu_s,m(K_{Y^\nu_s}+\Gamma^\nu_s))\arrow[l,"\bar{\varrho}_s^*",hook', two heads] & H^0(\bar{Z}_s,m\bar{A}_s)\arrow[l,"\bar{f}_s^*",hook', two heads]
    \end{tikzcd}    
\end{equation*}

Note that the morphisms $\varrho,\bar{\varrho}_s,f$, and $\bar{f}_s$ are $(K_X+\Delta)$-, $(K_{X_s}+\Delta_s)$-non-positive contractions, respectively $(K_Y+\Gamma)$-, $(K_{Y_s^\nu}+\Gamma_s^\nu)$-ample models, hence their induced pullback maps in the diagram above are all isomorphisms by \autoref{r-non_positive_contractions_preserve_sections}. By Serre vanishing, after possibly replacing $m$ with a multiple, we may assume that $\res_{Z_s}$ is surjective.

If $S$ is of equicharacteristic zero, $\nu$ and $h_s$ are isomorphisms by \autoref{l-pltcentersnormalUH} and \autoref{t-normalimageK_X+Dnonpositivefibration}, hence we conclude.

Suppose now $S$ is of positive or mixed characteristic. Let $\xi\in H^0(X_s,m(K_{X_s}+\Delta_s))$, and let $\bar{\xi}\in H^0(\bar{Z}_s,m\bar{A}_s)$ be the unique section such that $\xi=\bar{\varrho}_s^*\bar{f}_s^*\bar{\xi}$. Then $\xi$ extends if and only if $\bar{\xi}$ descends to $Z_s$. By \autoref{l-keelfrobeniusdescent} there exists $e\geq 0$, depending only on the morphism $h_s\colon \bar{Z}_s\to Z_s$ such that $\bar{\xi}^{p^e}$ descends to $Z_s$, for all $m$ and all $\bar{\xi}\in H^0(\bar{Z}_s,m\bar{A}_s)$.
\end{proof}


\begin{proof}[Proof of \autoref{t-extension_R^1_vanishing}]
    By \autoref{l-lifting_infinitesimals_implies_lifting_in_family} it is enough to consider the infinitesimal deformation case, which follows at once from \autoref{t-formallift_(X,L)_Lsa}, letting $\mathfrak{L}\coloneqq \omega_{X/R}^{[l]}(l\Delta)$ for some $l\geq 1$ sufficiently divisible.

    To prove items (a) and (b) it then suffices to show the vanishing $R^1f_{s,*}\cO_{X_s}$ under the corresponding hypotheses.
    
    In case (a), by \cite[Lemma 5.5]{Ejiri_Alb}, we have that $R^1f_{s,*}\cO_{X_s}$ is a vector bundle. Hence it suffices to show the vanishing $H^1(X_{s,y},\cO_{X_{s,y}})=0$, where $y\in Y_s$ is a general closed point. Since $f_s$ is locally $F$-split, the fiber $X_{s,y}$ is a globally $F$-split variety (\cite[Proposition 5.7]{ejiri}), hence the (twisted) Grothendieck trace map
    \[
    F_{X_{s,y},*}\cO_{X_{s,y}}((1-p^e)K_{X_{s,y}})\to\cO_{X_{s,y}}
    \]
    is split surjective for all $e\geq 0$ divisible enough. But $H^1(X_{s,y},(1-p^e)K_{X_{s,y}})=0$ for all $e\geq 0$ large enough by Serre vanishing, since $-K_{X_{s,y}}\sim_{\bQ}\Delta_{s,y}$ and the latter is ample, thus we conclude.

    Case (b) is immediate, since $Rf_{s,*}\cO_{X_s}$ is concentrated in degree zero by definition of rational singularities.
\end{proof}


\begin{proof}[Proof of \autoref{c-extension_gt_3folds}]
By \cite[Theorem 1.2]{HW} we can run a $(K_X+\Delta)$-MMP over $S$
\[
\varrho\colon (X,\Delta)\dashrightarrow (X^{\textup{m}},\Delta^{\textup{m}})
\]
which terminates with a minimal model. As $X$ is $\bQ$-factorial, so is $X^\Min$, in particular $(X^\Min,(X^\Min)_s)$ is plt. By \cite[Lemma 6.1]{HW} $(X^\Min)_s$ is normal\footnote{Note that \cite[Lemma 6.1]{HW} is stated in mixed characteristic but the same proof works in equicharacteristic $p>5$.}. In particular $\varrho_t$ is a $(K_{X_t}+\Delta_t)$-MMP for all $t\in S$, and $K_{X^\Min_s}+\Delta^\Min_s$ is big and semiample by the basepoint-free theorem (\cite[Theorem 1.1]{Xu_BPF}). Since MMP preserves $m$-pluricanonical sections for $m\geq 0$ divisible enough, and $\Proj(R(X_s,K_{X^\Min_s}+\Delta^\Min_s))$ has rational singularities by \cite[Corollary 1.3]{ABL}, we conclude by \autoref{t-extension_R^1_vanishing}(b).
\end{proof}


\begin{proof}[Proof of \autoref{t-bounded_3fold_canonical_models'}]

We first consider the case where $k$ is algebraically closed.

We begin by showing $\cC_v(k)$ is birationally bounded. First, note that if $X\in \cC_v(k)$ is such that $h^1(X,\cO_X)\neq 0$ and $Y\to X$ is a terminalization, then $h^1(Y,\cO_Y)\neq 0$ as well, since $X$ has rational singularities by \cite[Corollary 1.3]{ABL}. Fix now $m\geq 43$: if $X\in\cC_v(k)$, the map $\phi_{|mK_X|}$ is birational by \cite[Theorem 1.3]{zhang}, hence $X$ is birational to a subvariety $Z\subset \bP^N_k$ of degree $\leq m^3v$, where $N$ depends on $X$. After projecting from a general point we may assume that $Z\subset \bP^7_k$.
Let $\Univ_{n,d}(\bP_k^7)\to\Chow_{n,d}(\bP_k^7)$ be the universal family of cycles of dimension $n$ and degree $d$ on $\bP_k^7$\footnote{See \cite[Theorem 3.21]{KolRCAV}: note that the proof does not use the characteristic zero assumption until \cite[Definition 3.25.3]{KolRCAV}, and this is enough to construct the Chow variety together with the universal family for $\bP^N_k$.}. Then we have a birationally bounding family for $\cB$
\begin{equation}\label{e-chowfam}
\bigsqcup_{1\leq d\leq m^3v}\Univ_{3,d}(\bP^7_k)\to \bigsqcup_{1\leq d\leq m^3v}\Chow_{3,d}(\bP^7_k)
\end{equation}
More precisely: we get a surjective projective morphism of finite type $k$-schemes $\beta\colon\cY\to \cS$ such that for every $X\in \cC_v(k)$ there is $s\in \cS$ for which $\cY_s$ is a linear projection of $\phi_{|mK_X|}(X)$. Let $\cB$ be the class of these linear projections, so that $\beta\colon\cY\to\cS$ is a bounding family for $\cB$. By \autoref{l-simnor} we have a bounding family for $\cB^\nu\coloneqq \lbrace Z^\nu \textup{ such that }Z\in \cB\rbrace$,
\[
\bigsqcup_{i\in I}\beta_i\colon \bigsqcup_{i\in I}\cY_i\to \bigsqcup_{i\in I}\cS_i
\]
where each $\beta_i$ is a flat, surjective, projective morphism with geometrically integral and normal fibers of dimension three, each $\cY_i$ is normal, and each $\cS_i$ is smooth. By \autoref{l-simres} we may assume that every $\beta_i$ admits a simultaneous resolution $\cX_i\to \cY_i$. Let $\cX_i\dashrightarrow\cX_i^\Can$ denote the canonical model over $\cS_i$. After possibly replacing each $\cS_i$ with a suitable stratification, by \autoref{t-relcanstratification} we may further assume that $\cX_{i,s}\dashrightarrow(\cX_{i}^\Can)_s$ is the canonical model of $\cX_{i,s}$ for all $i\in I$ and all $s\in \cS_i$. Hence
\[
\bigsqcup_{i\in I}\cX^\Can_i\to \bigsqcup_{i\in I}\cS_i
\]
is a bounding family for $\cC_v(k)$.

Suppose now that $k$ is not algebraically closed and let $X\in \cC_v(k)$. For all $m\geq0 $ such that $mK_X$ is Cartier, we then have that
\begin{itemize}
    \item $H^i(X,mK_X)\otimes_k\bar{k}\cong H^i(X_{\bar{k}},mK_{X_{\bar{k}}})$ for all $i\geq 0$, by flat base change;
    \item $H^0(X,mK_X)$ induces an embedding if and only if $H^0(X_{\bar{k}},mK_{X_{\bar{k}}})$ induces an embedding, since being a closed immersion satisfies fpqc descent (\cite[Proposition 2.7.1]{EGAIV_2}).
\end{itemize}

Since the class $\lbrace X_{\bar{k}} \textup{ such that } X\in \cC_{3,v}(k)\rbrace$ is bounded, we can find $m_0\gg 0$ such that $H^i(X,mK_{X})=0$ for all $i>0$ and all $m\in m_0\bN$. By setting $P(l)\cong \chi(X,lm_0K_X)$ and $N\coloneqq h^0(X,m_0K_X)-1$, we have that the universal family on the Hilbert scheme
\[
\Univ_P(\bP^N_k)\to \textup{Hilb}_{P}(\bP^N_k)
\]
gives a bounding family for $\cC_v(k)$.
\end{proof}


\begin{proof}[Proof of \autoref{c-rational_points}]
    By \autoref{t-bounded_3fold_canonical_models'} we know that the class $\cC_v(k)$ is bounded. Let $\beta\colon \cX\to\cS$ be a bounding family. As $\beta$ is projective, we can find a finite surjective morphism $\cS'\to\cS$ such that the base change
    \[
    \beta'\colon \cX\times_{\cS}\cS'\to\cS'
    \]
    admits a section. By letting $d\coloneqq \deg(\cS'\to \cS)$, we conclude.
\end{proof}


\section{Plurigenera in projective families of smooth complex varieties}\label{s-plurigenera_smooth_complex_varieties}

This section is independent from the rest of the paper. The main result shows that, for projective families of \textit{complex} varieties, invariance of \textit{all} plurigenera is equivalent to invariance of \textit{all sufficiently divisible} plurigenera. We point out that this kind of result is completely false in positive and mixed characteristic: in \cite{Lang} and \cite{KU} the authors construct families of Enriques and elliptic surfaces where all sufficiently divisible plurigenera are invariant, but $h^0(X_s,K_{X_s})>h^0(X_\eta,K_{X_\eta})$. Even more surprisingly, in \cite{Suh} the author constructs examples of families of canonically polarized surfaces such that $h^0(X_s,K_{X_s})$ is arbitrarily larger than $h^0(X_\eta,K_{X_\eta})$.

Our proof closely follows \cite[Theorem 11.5.1]{La2}. We refer to \cite{La2} for the theory of multiplier ideal sheaves.

\begin{theorem}\label{t-AIP_iff_IP}
Let $C$ be a smooth integral complex affine curve, let $\pi\colon X\to C$ be a projective family of smooth varieties, and suppose that $h^0(X_c,mK_{X_c})$ is independent of $c\in C$ for all $m\geq 0$ divisible enough. Then $h^0(X_c,mK_{X_c})$ is independent of $c\in C$ for all $m\geq 0$.
\end{theorem} 

\begin{theorem}\label{t-lfness}
Let $C$ be a smooth integral complex affine curve, let $\pi\colon X\to C$ be a projective family of smooth varieties, and let $E$ be a Cartier divisor on $X$. Then 
$$R^q\pi_*\left( \cO_X(K_X+E)\otimes \cJ(\vvert{E}) \right)$$
is locally free for all $q\geq 0$.
\end{theorem}

\begin{proof}
Let $d\gg 0$ so that $\cJ(\vvert{E})=\cJ(d^{-1}|dE|)$. Let $f\colon Y\to X$ be a log resolution of $|dE|$, so that we can write $f^*|dE|=|U|+F$ where $U$ is free, and $F\geq 0$ is snc. Let $\delta\colon Y\to C$ be the induced morphism. By definition of multiplier ideal and the projection formula we have
\[R^q\pi_* \left( \cO_X(K_X+E) \otimes \cJ(\vvert{E}) \right)=\]
\[R^q\pi_*\left(f_*\cO_Y(f^*(K_X+E)+K_{Y/X}-\lfloor F/d\rfloor)\right).\]
Let $p\geq 1$: we then have
\[R^pf_*\cO_Y(f^*(K_X+E)+K_{Y/X}-\lfloor F/d\rfloor)=\]
\[\cO_X(K_X+E)\otimes R^pf_*\cO_Y(K_{Y/X}-\lfloor F/d\rfloor)=0,\]
where the first equality is given by the projection formula, and the second one by relative Nadel Vanishing \cite[Variant 9.4.5]{La2}. The Leray spectral sequence then yields
\[R^q\delta_*\cO_Y(f^*(K_X+E)+K_{Y/X}-\lfloor F/d\rfloor)=\]
\[R^q\pi_*\left( f_*\cO_Y(f^*(K_X+E)+K_{Y/X}-\lfloor F/d\rfloor)\right).\]
Let now $L\coloneqq f^*(K_X+E)+K_{Y/X}-\lfloor F/d\rfloor$: since $L\sim_{\bR}K_Y+\lbrace F/d \rbrace+U/d$, $\lbrace F/d \rbrace$ is snc, and $U/d$ is semiample, by Bertini's theorem (\cite[Lemma 4.7]{KolSOP}), we can write $\lbrace F/d \rbrace+U/d\sim_{\bR}B$ such that $(Y,B)$ is klt. Note that $R^q\pi_*\cO_Y(L)$ is a finitely generated module over the principal ideal domain $\cO_C$. Since $L-(K_Y+B)$ is $\pi$-semiample, by \cite[Theorem 6.3.(i)]{fujino2010} every nonzero section of $R^q\pi_*\cO_Y(L)$ contains $C$ in its support. In particular, $R^q\pi_*\cO_Y(L)$ has no torsion sections, thus it is locally free.
\end{proof}

Next is an asymptotic version of the Restriction Theorem for ordinary multiplier ideals.

\begin{proposition}\label{l-asymptotic_restriction}
Let $C$ be a smooth integral complex affine curve, let $\pi\colon X\to C$ be a projective family of smooth varieties, and let $L$ be a Cartier divisor on $X$. Then $\cJ(\vvert{mL}_c)\subset \cJ(\vvert{mL})_{X_c}$ for all $m\geq 0$ and all $c\in C$.
\end{proposition}

\begin{proof}
The statement is obvious when $m=0$. Fix then $m\geq 1$ and $c\in C$, and consider the graded linear series 
$$\lbrace V_p\coloneqq \im\left[H^0(X,pmL)\to H^0(X_c,pmL_c)\right]\rbrace_{p\in\bN}.$$
Let $p\gg 0$ so that
$$\cJ(\vvert{mL}_c)=\cJ(p^{-1}|V_p|)\hspace{5mm}\cJ(\vvert{mL})=\cJ(p^{-1}|pmL|).$$
Let $D\in |pmL|$ be a sufficiently general divisor so that $D_c\coloneqq D|_{X_c}\in V_p$ is also general.  By \cite[Proposition 9.2.26]{La2} we have
$$\cJ(p^{-1}|V_p|)=\cJ(p^{-1}D_c)\hspace{5mm}\cJ(p^{-1}|pmL|)=\cJ(p^{-1}D),$$
thus we conclude by \cite[Theorem 9.5.1]{La2}.
\end{proof}

\begin{corollary}\label{c-OT}
Let $C$ be a smooth integral complex affine curve, let $\pi\colon X\to C$ be a projective family of smooth varieties, and let $c\in C$ be a closed point. If $\sigma_c\in H^0(X_c,mK_{X_c})$ vanishes along the ideal $\cJ(\vvert{(m-1)K_X})_{X_c}$, then there exists $\sigma\in H^0(X,mK_X)$ such that $\sigma\vert_{X_c}=\sigma_c$.
\end{corollary}

\begin{proof}
The sheaves $R^q\pi_*\left(\cJ(\vvert{(m-1)K_X})\otimes\cO_X(mK_X)\right)$ are locally free for all $q\geq 0$ by \autoref{t-lfness}, thus we conclude by \cite[Theorem III.12.11]{Har}.
\end{proof}

\begin{proof}[Proof of \autoref{t-AIP_iff_IP}]
First of all, note that the hypothesis implies that $\kappa\coloneqq \kappa(K_{X_c})$ is independent of $c\in C$. If $h^0(X_c,mK_{X_c})=0$ for all $c\in C$ and all $m\geq 0$ divisible enough, then $\kappa=-\infty$ for all $c\in C$, thus $h^0(X_c,mK_{X_c})=0$ for all $m>0$ and all $c\in C$. Thus we can assume $0\leq \kappa\leq\dim (\pi)$. Note that $h^0(X_c,K_{X_c})$ a Hodge number, hence it is invariant in smooth projective families. Let $m\geq 2$ and let $\sigma_c\in H^0(X_c,mK_{X_c})$: by \autoref{c-OT} it is enough to show that $\sigma_c$ vanishes along $\cJ(\vvert{(m-1)K_X})_{X_c}$.

Let $e\in\bN$ such that $h^0(X_c,ekK_{X_c})$ is independent of $c\in C$ for all $k\in\bN$, and consider the equation
\begin{equation}\label{e-trick}
m(l+1)=ek.
\end{equation}
There are infinitely many positive integer values of $k$ for which \autoref{e-trick} has a solution in $l$. For all such $l$, consider the section $\sigma_c^{l+1}\in H^0(X_c,m(l+1)K_{X_c})$. Since $h^0(X_c,m(l+1)K_{X_c})$ does not depend on $c\in C$, the restriction map
$$H^0(X,m(l+1)K_X)\to H^0(X_c,m(l+1)K_{X_c})$$
is surjective, thus $\sigma^{l+1}_c$ vanishes along $\cJ(\vvert{m(l+1)K_X}_c)$. We have inclusions
\[
\begin{split}
\cJ(\vvert{m(l+1)K_X}_c)&\subset\cJ(\vvert{mK_X}_c)^{l+1}\\
							&\subset\cJ(\vvert{mK_X}_c)^l\\
							&\subset\cJ(\vvert{(m-1)K_X}_c)^l
\end{split}
\]
where the first one comes from \cite[Corollary 11.2.4]{La2} and the third one from \cite[Theorem 11.1.19]{La2}. Hence $\sigma_c^{l+1}$ vanishes along $\cJ(\vvert{(m-1)K_X}_c)^l$ for infinitely many $l$. By \cite[Example 11.5.6]{La2} $\sigma_c$ vanishes along the integral closure of $\cJ(\vvert{(m-1)K_X}_c)$, which coincides with $\cJ(\vvert{(m-1)K_X}_c)$, since multiplier ideals are integrally closed by \cite[Corollary 9.6.13]{La2}. By \autoref{l-asymptotic_restriction} $\cJ(\vvert{(m-1)K_X}_c)\subset \cJ(\vvert{(m-1)K_X})_{X_c}$, thus we conclude.
\end{proof}

\section{An example}\label{s-example}

Throughout this section $S$ will denote the spectrum of a DVR of residue characteristic $p>2$, with closed and generic point $s$ and $\eta$, respectively.

\begin{theorem}\label{t-AIP_fails_bfree}
    There are projective families of smooth 3-folds of Kodaira dimension one $\pi\colon X\to S$, such that $K_X$ is semiample and asymptotic invariance of plurigenera fails.
\end{theorem}

\begin{remark}\label{r-failureAIP_Pairs}
    This kind of strong failure of invariance of plurigenera has been first observed in \cite{Bri20} and \cite{Kol3foldcharp} for families of good minimal models of terminal surface pairs of Kodaira dimension one and  of plt 3-folds of general type, respectively. However, all the previously known examples needed a non-empty boundary $\Delta\neq 0$. \autoref{t-AIP_fails_bfree} is new in this sense. 
\end{remark}

\begin{proof}
    We follow closely the construction in \cite{Bri20}. First, consider a family of semistable elliptic K3 surfaces, i.e. a commutative diagram
    \begin{equation*}
        \begin{tikzcd}
            W\arrow[dr,"\psi",swap]\arrow[rr,"\varphi"] &         & \bP^1_S\arrow[dl]\\
              & S, &
        \end{tikzcd}
    \end{equation*}
    where $\psi$ is a smooth projective morphism whose fibers are K3 surfaces, for all $t\in S$ the morphism $\varphi_t$ is a contraction with genus one fibers, and its only singular fibers are of type $I_n$. We will denote by $P_1,\dots,P_l\in W(S)$ the singular points of the morphism $\varphi$, and assume that they are pairwise disjoint. We will also assume that $\varphi$ is smooth over $0,\infty\in \bP^1_S(S)$. Such a family can be easily constructed by ``spreading out'' one of the examples in \cite{MP} over a sufficiently large mixed characteristic base, and then reducing it modulo a sufficiently large prime.

    Let $E\to S$ be a smooth projective family of genus one curves, admitting a non-trivial $q=p^e$-torsion line bundle $M$, such that $M_s\cong \cO_{E_s}$. Consider the product $Z\coloneqq W\times_S E$ with the line bundle $L\coloneqq \varphi^*\cO_{\bP^1_S}(1)\boxtimes M$. Let $[U:V]$ be homogeneous coordinates on $\bP^1_S$ and let $\sigma\coloneqq \varphi^*(V^{q-1}U)\boxtimes\mathbf{1}\in H^0(Z,L^q)$, where $\mathbf{1}\in H^0(E,M^{q})$ is a nowhere vanishing section. 
    Consider now the following chain of morphisms
    \[
    Y\coloneqq (Z[\sigma^{1/q}])^{\nu}\xrightarrow{\gamma} Z\to W\to \bP^1_S,
    \]
    where $\gamma$ is the normalized $q$-to-1 cyclic cover branched along $(\sigma=0)$. Over the affine local charts $(V\neq 0)$ and $(U\neq 0)$ in $\bP^1_S$, the above composition takes the form
    \begin{equation}\label{e-local_coord}
        \begin{split}
            \frac{\cO_{Z}[\lambda]}{(\lambda^q-\alpha\cdot \varphi^*(u))}\leftarrow \cO_Z\to \cO_W\xleftarrow{\varphi^*} \cO_S[u],\\
            \frac{\cO_{Z}\left[\varphi^*(v)/\xi,\xi \right]}{(\xi-\beta\cdot (\varphi^*(v)/\xi)^{q-1})}\leftarrow \cO_Z\leftarrow \cO_W\xleftarrow{\varphi^*} \cO_S[v],
        \end{split}
    \end{equation}
    respectively. Here $\alpha$ and $\beta$ are sections of $\cO_E^\times$ corresponding to local trivializations of the section $\mathbf{1}\in H^0(E,M^q)$. In particular, $\alpha_s,\beta_s\in(\cO_{E_s}^\times)^q$. Hence we see that $Y$ is smooth over $S$ away from $\gamma^{-1}(\lbrace P_j\rbrace_{j=1}^l\times_S E)$. Formally-locally around $\lbrace P_j\rbrace \times_S E$, the above composition becomes
    \[
    \frac{\cO_E\llbracket x,y\rrbracket[\lambda]}{(\lambda^q-\alpha\cdot xy)}\leftarrow \frac{\cO_E\llbracket x,y,u\rrbracket}{(u-xy)}\leftarrow \frac{\cO_S\llbracket x,y,u\rrbracket}{(u-xy)}\leftarrow \cO_S\llbracket u\rrbracket.
    \]
    In particular, as $p>2$, the singularities are \'etale-locally isomorphic to $A_{q-1}\times E$, where $A_{q-1}$ denotes the DuVal singularity $(\lambda^q-xy=0)$. By \cite{Artin}, after possibly replacing $S$ by a finite extension, we can find a simultaneous crepant resolution $\widetilde{A}_{q-1}\to A_{q-1}$ over $S$. As blowups can be constructed \'etale-locally (\cite[Tag 085S]{SP}) we have an induced simultaneous crepant resolution $r\colon X\to Y$ over $S$. 
    As $K_Z\sim 0$ we have that $K_Y\sim_{\bQ}\gamma^*((q-1)L)$ is semiample. As $r_*\cO_X=\cO_Y$ and $r$ is crepant, we have that $K_X$ is semiample as well. In particular, its Iitaka fibration is the induced morphism
    \[
    f\colon X\xrightarrow{r} Y \xrightarrow{\gamma} Z\to \bP_S^1.
    \]
    Thus, by \autoref{l-stein_invariance+}, asymptotic invariance of plurigenera holds if and only if the restriction $f_s$ is a contraction. If that were the case, then the geometric generic fiber of $f_s$ ought to be reduced (see \cite[Theorem 2]{Maclane}). But from \autoref{e-local_coord} we see that this is not the case.
\end{proof}
\begin{remark}
    A simple computation using the projection formula shows that $h^0(X_s,mK_{X_s})\approx p^e\cdot h^0(X_\eta,mK_{X_\eta})$ for all sufficiently divisible $m\geq 1$. In particular, for all primes $p>2$, we can obtain projective families of smooth good minimal models with arbitrarily big jumps in plurigenera.
\end{remark}

\bibliographystyle{alpha}
\bibliography{bib.bib}

\end{document}